\documentclass{amsart}
\usepackage{amssymb}
\usepackage{amscd}
\usepackage[latin1]{inputenc}
\usepackage{amsmath}
\usepackage{amsfonts}
\usepackage{latexsym}
\newtheorem{theorem}{Theorem}[section]
\newtheorem{corollary}[theorem]{Corollary}

\newtheorem{definition}[theorem]{Definition}

\newtheorem{proposition}[theorem]{Proposition}
\newtheorem{example}[theorem]{Example}
\theoremstyle{remark}

\newtheoremstyle{nonum}{}{}{\upshape}{}{\itshape}{.}{ }{#1#3}
\theoremstyle{nonum}
\newtheorem{remark*}{Remark}
\newcommand{\dbar}{d\mkern-6mu\mathchar'26}

\newcommand{\comment}[1]{}
\def \beq {\begin{eqnarray}}
\def \eeq {\end{eqnarray}}
\def \beqn {\begin{eqnarray*}}
\def \eeqn {\end{eqnarray*}}

\def\exp{{\mbox{exp}}}
\def\A{{\mathcal A }}

\def\C{{\mathcal C }}

\def\E{{\mathcal E }}

\def\L{{\mathcal L }}
\def\H{{\mathcal H }}

\def\zz{{\mathbb{Z}}}

\begin{document}
\title[Approximating entropy and pressure for
$\zz^2$ MRFs and SFTs]{Approximating entropy for a class of $\zz^2$
Markov Random Fields and pressure for a class of functions on
$\zz^2$ shifts of finite type}

\begin{abstract}
For a class of $\zz^2$ Markov Random Fields (MRFs) $\mu$, we show
that the sequence of successive differences of entropies of induced MRFs on
strips of height $n$ converges exponentially fast (in $n$) to the
entropy of $\mu$.  These strip entropies can be computed explicitly when $\mu$ is a Gibbs state
given by a nearest-neighbor interaction on a strongly irreducible
nearest-neighbor $\zz^2$ shift of finite type $X$.
We state this result in terms of approximations to the (topological) pressures
of certain functions on such an $X$, and we show that these pressures are computable
if the values taken on by the functions are computable. Finally, we show
that our results apply to the hard core model and Ising model for certain
parameter values of the corresponding interactions, as well as to the
topological entropy of certain nearest-neighbor $\zz^2$ shifts of finite type,
generalizing a result in~\cite{Pa}.
\end{abstract}


\date{}
\author{Brian Marcus}
\address{Brian Marcus\\
Department of Mathematics\\
University of British Columbia\\
1984 Mathematics Road\\
Vancouver, BC V6T 1Z2}
\email{marcus@math.ubc.ca}
\author{Ronnie Pavlov}
\address{Ronnie Pavlov\\
Department of Mathematics\\
University of British Columbia\\
1984 Mathematics Road\\
Vancouver, BC V6T 1Z2} \email{rpavlov@du.edu}
\renewcommand{\subjclassname}{MSC 2000}
\maketitle

\section{Introduction}\label{intro}

The concept of entropy is fundamental to the study of dynamical
systems both in topological dynamics, where it arises as
topological entropy for continuous maps, and in ergodic theory, where
it arises as measure-theoretic entropy for measure-preserving
transformations.

Of particular interest in symbolic dynamics are dynamical systems known
as shifts of finite type. We restrict our attention to
nearest neighbor $\zz^d$ shifts of finite type (n.n. $\zz^d$-SFT); such an SFT $X$
is specified by a finite alphabet $\A$ and a set of translation-invariant adjacency rules:
$X$ is the subset of $\A^{\zz^d}$ of all
configurations on $\zz^d$ which satisfy the adjacency rules.
Here, the underlying dynamics are given by the group of
translations by vectors in $\zz^d$.
The topological entropy $h(X)$ is defined as the
asymptotic growth rate of the number of configurations on finite
rectangles that extend to
elements of $X$ (more precise definitions for n.n. $\zz^d$-SFT, topological entropy
and other concepts used in this introduction
are given in Section \ref{defns}).

%

The most prominent non-trivial example in dimension $d=1$ (i.e. n.n. $\zz$-SFT)
is the golden mean shift, defined as the set of all bi-infinite $0-1$
sequences that do not contain two adjacent $1$'s. Its two-dimensional
analogue, known as the {\em hard square shift} $\mathcal{H}$, is
defined as the set of all $0-1$ configurations on $\zz^2$ such that
$1$'s are never adjacent horizontally or vertically.

The topological entropy of a n.n. $\zz$-SFT $X$ is easy to
compute: namely, $h(X)$ is the log of the largest eigenvalue of a
nonnegative integer matrix defined by the restricted adjacency
rules.  The topological entropy of the golden mean shift turns out
to be the $\log$ of the golden mean (hence the name for this SFT).
However,  it is very difficult in general to compute the
topological entropy of a n.n. $\zz^2$-SFT, and exact
values are known in only a handful of cases. Even for the hard
square shift, the topological entropy is not known.

Given a n.n. $\zz^2$-SFT $X$, the allowed configurations
on a strip of height $n$ form what is effectively a n.n.
$\zz$-SFT $X_n$; here, the alphabet consists of columns of height $n$
that obey the vertical adjacency rules, with two adjacent columns
required to satisfy the horizontal adjacency rules.
In~\cite{Pa}, Pavlov proved
that for the hard square shift $X = \mathcal{H}$, the sequence of
differences $h(X_{n+1}) - h(X_n)$ not only converges to $h(X)$ but does
so exponentially fast (as a function of $n$).  While this does not give an
exact expression for $h(X)$, it
does show that $h(X)$ can be approximated relatively well.  In
particular, a consequence of the approximation result from~\cite{Pa} is that there is a
polynomial time algorithm which
on input $k$ produces an estimate of $h(\mathcal{H})$ guaranteed to be accurate within
$1/k$.  This is in stark contrast to the main result from~\cite{HM}, which implies that
there exist numbers which occur as the topological entropy of a
n.n. $\zz^2$-SFT and are arbitrarily poorly computable.

While topological entropy can be viewed as a purely combinatorial object, the proof in~\cite{Pa}
uses measure-theoretic tools. For a translation-invariant
measure $\mu$ on $\A^{\zz^2}$, there is an analogous notion of measure-theoretic entropy
$h(\mu)$.
If the support of $\mu$ is contained in a n.n.
$\zz^2$-SFT $X$, then $h(\mu) \le h(X)$ and there is always at least one
measure $\mu$ such that $h(\mu) = h(X)$. For $X =
\mathcal{H}$, there is a unique measure of maximal entropy
$\mu_{\rm max}$. There is also a unique measure of maximal entropy
$\mu_n$ for each ${\mathcal H}_n$.  Using results on
``disagreement percolation'' from~\cite{vdBS}, it was shown
that the sequence of differences
$h(\mu_{n+1})- h(\mu_n)$
converges exponentially fast to $h(\mu_{\rm max})$,
and one concludes that $h({\mathcal{H}}_{n+1}) - h(\mathcal{H}_n)$ converges exponentially fast to $h(\mathcal{H})$,
as desired.

%

%
%

Of critical importance to the proof
is the fact that $\mu_{max}$ is a $\zz^2$-Markov random field (MRF), which is,
roughly speaking, a measure on $\A^{\zz^2}$ such that for any finite subset $S \subset \zz^2$
the conditional probability distribution on configurations on $S$ given a configuration $u$ on $\zz^2 \setminus S$
depends only on the restriction of $u$ to the boundary of $S$.
For $\mu_{max}$, these conditional probabilities
are uniform over allowed configurations on $S$ given $u$.

In this paper, we generalize the main result of~\cite{Pa} to more general translation-invariant
$\zz^2$-MRFs $\mu$, with non-uniform
conditional probabilities.
Using results from~\cite{vdBM}, instead of~\cite{vdBS}, we give, in
Section \ref{MRFuniqueness}, a measure-theoretic analogue for
certain $\zz^2$-MRFs.  Namely, our Theorem~\ref{maintheorem2} asserts that for sufficiently large $n$,
the MRF $\mu$ induces MRFs $\mu_n$ on
strips of height $n$, with appropriate boundary conditions, such that
the sequence of differences $h(\mu_{n+1}) - h(\mu_n)$
converges exponentially fast to $h(\mu)$ (here, $\mu_n$ is viewed as a one-dimensional stationary process
on sequences of configurations of $n$-high columns);
this result requires the
existence of suitable boundary rows for sufficiently large $n$ and a
condition on the probability distributions on configurations at a
site in $\zz^2$, conditioned on configurations on its four nearest neighbors.
The condition (which is from \cite{vdBM}) is that any two such (conditional) probability
distributions should not be too different -- more precisely, the
total variation distance between any such distributions should be
less than the critical value for site percolation in $\zz^2$.
This condition is similar in spirit to the classical Dobrushin uniqueness criterion.
%
The induced MRF $\mu_n$ is defined by restricting
the conditional probability specifications of $\mu$ to the
strip of height $n$, with appropriate boundary conditions imposed on the row
immediately above the top
row of the strip and the row immediately below the
bottom row of the strip. 
We note, in particular, that $\mu_n$ is {\em not} the usual marginalization of $\mu$ to the strip;
this latter process is typically not even an MRF.

%

This all becomes more concrete when the MRF $\mu$ is a Gibbs state for a n.n.
interaction $\Phi$ on a n.n. $\zz^2$-SFT $X$ which satisfies a strong irreducibility condition
(Gibbs states are discussed in Section~\ref{gibbs} and the strong irreducibility condition
and consequences are discussed in Section~\ref{SI}).
In this case, the induced MRFs $\mu_n$ are translation-invariant first-order Markov chains
whose transition probabilities are easily computed from the interaction (Proposition~\ref{Markov} in Section~\ref{markov}).
There is a simple closed form for the entropy of such a Markov chain, which
in spirit is similar to the closed form for topological entropy of a n.n. $\zz$-SFT.

The interaction $\Phi$ defines a continuous function $f_\Phi$ on $X$. The
pressure $P_X(f_\Phi)$ of such a  function is defined as the asymptotic growth rate of arrays which are,
roughly speaking, weighted by $\exp(f_\Phi)$.  Using an equivalent variational formula for
$P_X(f_\Phi)$, given in terms of translation-invariant measures supported within $X$, one can apply
%
Theorem~\ref{maintheorem2} to obtain
exponentially fast approximations to pressures of such functions $f_\Phi$
by differences of pressures of induced functions on strips of height $n$ (Theorem~\ref{main_Gibbs2});
these pressures are computed as largest eigenvalues of explicit matrices.
A corollary of this result is
Theorem~\ref{pressurecomputability} which expresses the
computability of $P_X(f_\Phi)$ in terms of computability of the values of $\Phi$.

Finally, examples of Gibbs states and corresponding pressures are given in Section~\ref{examples}.
We consider two classical examples: the two-dimensional hard core model,
given by an interaction parameterized by activity level $a$, and the two-dimensional Ising
antiferromagnet model, parameterized by
inverse temperature $\beta$ and external field $h$.  Explicit ranges of values of
these parameters are given for which Theorem~\ref{main_Gibbs2} applies.

When $\Phi = 0$, then $f_\Phi = 0$ and $P_X(f_\Phi)$ reduces to $h(X)$.  It follows that the approximation
result for pressure (Theorem~\ref{main_Gibbs2})
can be used to obtain exponentially fast approximations to $h(X)$ for certain n.n. $\zz^2$-SFTs. In particular,
this result recovers the main result of~\cite{Pa} and extends that result to other
n.n. $\zz^2$-SFTs, examples of which are given in Section~\ref{examples}.

%
%

\section{Definitions and preliminaries}
\label{defns}

An {\em undirected graph} $G$ consists of a set of {\em vertices}
(or {\em sites}) $V(G)$ and a set of {\em edges} (or {\em nearest
neighbors}) $E(G)$ of (unordered) pairs of distinct vertices.
All graphs we consider will be countable and locally finite.
Two vertices $v,w \in V(G)$ are said to be {\em adjacent} if $\{v,w\} \in E(G)$.
For finite sets $U_1, U_2 \subset V(G)$, let $E(U_1,U_2)$ denote the
set of all edges in $G$ with one vertex in $U_1$ and the other in
$U_2$.

For any $d > 0$, we (in a slight abuse of notation) use $\mathbb{Z}^d$ to denote the \textit{$d$-dimensional cubic lattice}, the graph defined by $V(\mathbb{Z}^d) = \zz^d$ and $E(\mathbb{Z}^d) = \{\{u,v\} \ : \ \sum_{i=1}^d |u_i - v_i| = 1\}$.

The \textit{boundary} of a set $S \subset V(G)$ within a graph $G$, which is denoted by $\partial(S,G)$, is the set of $v \in V(G) \setminus S$ which are adjacent to some element of $S$. If we refer to simply the boundary of a set $S$, or write $\partial S$, then the graph $G$ is assumed to be $\mathbb{Z}^2$. In the case where $S$ is a singleton $\{v\}$, we call the boundary the {\em set of neighbors} $N^G_v$, which is just the set of $w \in V(G)$ adjacent to $v$. Again, when no mention of $G$ is made, it is assumed to be $\zz^2$.

For any integers $a < b$, we use $[a,b]$ to denote $\{a,a+1,\ldots,b\}$.

\

An \textit{alphabet} $\A$ is a finite set with at least two elements.

A \textit{configuration} $u$ on the alphabet $\A$ in the graph $G$ is any mapping from a non-empty subset $S$ of $V(G)$ to $\A$, where $S$ is called the \textit{shape} of $u$. For any configuration $u$ with shape $S$ and any $T \subseteq S$, denote by $u|_T$ the restriction of $u$ to $T$, i.e. the subconfiguration of $u$ occupying $T$. For $S,T$ disjoint sets, $x \in \A^S$ and $y \in \A^T$, $xy$ denotes the configuration on $S \cup T$ defined by $(xy)|_S = x$ and $(xy)|_T = y$, which we call the \textit{concatenation} of $x$ and $y$.

For any $d$, we use $\sigma$ to denote the natural {\em shift
action} on $\A^{\mathbb{Z}^d}$ defined by $(\sigma_{v}(x))(u) =
x(u+v)$.

For any alphabet $\A$ and graph $G$, $\A^{V(G)}$ is a topological space when endowed
with the product topology (where $\A$ has the discrete topology), and any subsets will
inherit the induced topology. We will also frequently speak of measures on $\A^{V(G)}$,
and all such measures in this paper will be Borel probability measures. This means that any $\mu$ is determined by its values on the sets $[w] := \{x \in \A^{V(G)} \ : \ x|_S = w\}$, where $w$ is a configuration with arbitrary finite shape $S \subseteq V(G)$. Such sets are called {\em cylinder sets}, and for notational convenience, rather than referrring to a cylinder set $[w]$ within a measure or conditional measure, we just use the configuration $w$. For instance, $\mu(w \cap v \ | \ u)$ represents the conditional measure $\mu([w] \cap [v] \ | \ [u])$.

A measure $\mu$  on $\A^{\zz^d}$ is {\em translation-invariant} (or
{\em stationary})  if $\mu(A) = \mu(\sigma_{v} A)$ for all
measurable sets $A$ and $v \in \zz^d$. A translation-invariant
measure $\mu$ on $\A^{\mathbb{Z}^d}$ is {\em ergodic} if whenever $U
\subseteq \A^{\mathbb{Z}^d} $ is measurable and translation-invariant, then
$\mu(U) = $ 0 or 1.

\

Let $d$ be a positive integer. Let $\E_1, \ldots, \E_d \subseteq
\A^2$. The {\em nearest neighbor $\zz^d$ shift of finite type (n.n. $\zz^d$-SFT)}
$X$, defined by $\E_1, \ldots, \E_d$ is the set $X$ of all $x \in
\A^{\zz^d}$ such that whenever $u\in \zz^d$ and $1 \le i \le d$,
~we have $x(u) x(u + e_i) \in \E_i$, where $e_i$ is the $i$th standard
basis vector. We say that $X$ is a n.n. SFT if it is a n.n. $\zz^d$-SFT for some $d$.

When $d=1$, we write $\E = \E_1$. Any n.n. $\zz$-SFT $X$ defined by $\E$ has an associated square $|\A| \times |\A|$ matrix $A$, called the \textit{adjacency matrix}, which is defined by $a_{i,j} = \chi_{\E}(i,j)$. In other words, $a_{i,j} = 1$ iff $(i,j) \in \E$.

The {\em language} of $X$ is:
$$
\L(X) = \cup_{\{S \subset \zz^d, ~ |S| < \infty\}} \L_S(X)
$$
where
$$
\L_S(X) = \{x|_S: x \in X\}.
$$

For a subset $S \subseteq \zz^d = V(\mathbb{Z}^d)$, finite or infinite, a configuration
$x \in \A^S$ is {\em globally admissible} for $X$ if $x$ extends to a configuration on all of
$\zz^d$. So, the language $\L(X)$ is precisely the set of globally
admissible configurations on finite sets.

A configuration $x \in \A^S$ is {\em locally admissible} for $X$ if for all edges $e = \{u, u
+ e_i\}$ contained in $S$, we have $x|_e \in \E_i$. We note that technically there is an ambiguity
here since several choices for $\mathcal{E}_i$
could induce the same n.n. SFT. For this reason, we will always think
of a n.n. SFT as being ``equipped'' with a specific choice of the sets $\mathcal{E}_i$.

\begin{example}
The $\mathbb{Z}^2$ {\rm hard square shift} $\mathcal{H}$ is the n.n. shift of finite type with alphabet $\{0,1\}$ and $\E_1 = \E_2 = \{(0,0), (0,1), (1,0)\}$. 
\end{example}

Given any measure $\mu$ on $\A^{\zz^d}$ and any rectangular prism $R = \prod {[1,n_i]}$, we can associate a \textit{$R$-higher power code} $\mu^{[R]}$ of $\mu$, defined as the image of $\mu$ under the mapping $\phi_R: \A^{\zz^d} \rightarrow (\A^R)^{\zz^d}$ defined by $(\phi_R x)(v_1, \ldots, v_d) = x|_{\prod [n_i v_i, n_i (v_i + 1) - 1]}$. For any n.n. $\zz^d$-SFT $X$, $\phi_R(X)$ is also a n.n. $\zz^d$-SFT, which we call the \textit{$R$-higher power code} of $X$ and denote by $X^{[R]}$.

We define a square nonnegative matrix $A$ to be \textit{primitive} if some power $A^n$ has all positive entries. This allows us to define the notion of mixing for two types
of one-dimensional dynamical systems. A n.n. $\zz$-SFT $X$ is called \textit{mixing} iff its adjacency matrix (after discarding any letters of $\A$ which do not actually appear in $X$) is primitive, and a Markov chain is called \textit{mixing} if its transition probability matrix is primitive.

For any translation-invariant measure $\mu$ on $\A^{\mathbb{Z}^d}$, we may define its entropy as
follows.

\begin{definition}
The {\rm measure-theoretic entropy} of a translation-invariant
measure $\mu$ on $\A^{\mathbb{Z}^d}$ is
defined by
\[
h(\mu)=\lim_{j_1, j_2, \ldots, j_d \rightarrow \infty} \frac{-1}{j_1 j_2 \cdots j_d} \sum_{w \in \A^{\prod_{i=1}^d [1,j_i]}} \mu(w) \log(\mu(w)),
\]

\noindent
where terms with $\mu(w) = 0$ are omitted.
\end{definition}

We will also deal with measure-theoretic conditional entropy in this paper. It can be defined more generally, but for our purposes, we will define it only for a measure on $\A^{\zz}$ and specific type of partition of $\A^{\zz}$. For any partition $\xi$ of a set $S$, and for any $s \in S$, we use $\xi(s)$ to denote the element of $\xi$ which $s$ is in. If $\xi$ is a partition of an alphabet $\A$, then $\phi_{\xi}$ is the map on $\A^{\zz}$ defined by $\phi_{\xi}(x) = \ldots \xi(x_{-1}) \xi(x_0) \xi(x_1) \ldots$.

We note that for any measure $\mu$ on $\A^{\zz}$ and any partition $\xi$ of $\A$, the push-forward $\phi_{\xi}(\mu)$ of $\mu$ under the map $\phi_{\xi}$ is a measure on $\xi^{\zz}$.

\begin{definition}
For any translation-invariant measure $\mu$ on $\A^{\zz}$ and any partition $\xi$ of $\A$, the {\rm conditional measure-theoretic entropy} of $\mu$ with respect to $\xi$ is
\[
h(\mu \ | \ \xi) = \lim_{k \rightarrow \infty} \frac{-1}{2k+1} \sum_{w \in \A^{[-k,k]}} \mu(w) \log\bigg(\frac{\mu(w)}{(\phi_{\xi}(\mu))\big(\xi(w_{-k}) \ldots \xi(w_k)\big)}\bigg),
\]
where again terms with $\mu(w) = 0$ are omitted.
\end{definition}

Measure-theoretic conditional entropy is most useful because of the following decomposition formula. For a proof, see \cite{Pa}.

\begin{proposition}\label{P2}
For any translation-invariant measure $\mu$ on $\A^\zz$, and any partition $\xi$ of $\A$,
\[
h(\mu \ | \ \xi) = h(\mu) - h(\phi_{\xi}(\mu)).
\]
\end{proposition}

The {\em weak topology} on the space of measures on $\A^\zz$ is the weakest topology under which
integrals of real-valued continuous functions converge.
Measure-theoretic (and conditional measure-theoretic) entropy are not continuous in the weak topology
(though they are upper semicontinuous); see \cite{walters}. For this reason, we need to define the $\dbar$
metric for measures, with respect to which the entropy map $\mu \mapsto h(\mu)$ is continuous
and in fact H\"{o}lder. We first need the preliminary definition of
a coupling.

\begin{definition}
For any measures $\mu$ on $X$ and $\nu$ on $Y$, a {\rm coupling} of $\mu$ and $\nu$ is a measure $\lambda$ on $X \times Y$ for which $\lambda(A \times Y) = \mu(A)$ for any $\mu$-measurable $A \subseteq X$ and $\lambda(X \times B) = \nu(B)$ for any $\nu$-measurable $B \subseteq Y$. The set of couplings of $\mu$ and $\nu$ is denoted by $C(\mu,\nu)$.
\end{definition}

\begin{definition}
For any measures $\mu$ and $\mu'$ on $\A^{\mathbb{Z}}$,
\[
\dbar(\mu,\mu') = \limsup_{n \rightarrow \infty} \min_{\lambda \in C(\mu|_{[-n,n]},\mu'|_{[-n,n]})} \int d_{2n+1}(x,y) \ d\lambda,
\]

\noindent
where $d_k$ is the
normalized $k$-letter Hamming distance between $k$-letter configurations given by $d_k(u,v) = \frac{1}{k} \sum_{1 \leq i \leq k} (1 - \delta_{u(i) v(i)})$.

\end{definition}

We briefly summarize some important properties of the $\dbar$
distance (for
more information, see \cite{rudolph} or \cite{Sh}).  We are
interested only in the $\dbar$ metric on the space of
translation-invariant measures on $\A^{\mathbb{Z}}$ for a fixed
alphabet $\A$. The $\dbar$ metric is complete and  dominates
distribution distance in the sense that for any configuration $w$ on
a finite interval of length $m$, $|\mu(w) - \mu'(w)| \le
m\dbar(\mu,\mu')$.
And H\"{o}lder continuity of entropy follows from the estimate:
letting $\epsilon = \dbar(\mu,\mu')$,
$$
|h(\mu) - h(\mu')| \leq \epsilon \log |\A| - \epsilon \log \epsilon
- (1 - \epsilon) \log(1 - \epsilon);
$$
(see [\cite{rudolph}, Theorem 7.9] for a proof in the ergodic case;
the same estimate holds in the general translation-invariant case).
%
%
%

Finally, we define the topological pressure of a continuous function
on a n.n. SFT, following \cite{ruelle}.

Let $X$ be a n.n. $\zz^d$-SFT, and let $f: X \rightarrow \mathbb{R}$ be a continuous function.
The topological pressure of $f$ on $X$ can be defined in several ways; one is as the purely topological notion of the asymptotic growth rate of the number of (locally or globally) admissible arrays in $X$, ``weighted by $f$.''
For our purposes though, the following definition, which is a consequence of the variational principle
(see \cite{Mi} for a short proof), is more convenient.

\begin{definition}
Given a n.n. $\zz^d$-SFT $X$ and $f \in C(X)$, the {\rm (topological) pressure} of $f$ on $X$ is:
$$
P(f) = P_X(f) = \sup_\mu \left(h(\mu) + \int f d\mu\right),
$$
where the $\sup$ is taken over all translation-invariant measures $\mu$ supported on $X$.
\end{definition}

The $\sup$ is always achieved and any measure which achieves the $\sup$
is called an {\em equilibrium state} for $X$ and $f$. (\cite{walters})

In the special case when $f = 0$, $P(f)$ is called the \textit{topological entropy} $h(X)$ of $X$,
and any equilibrium state is called a \textit{measure of maximal entropy} for $X$.

\section{Exponential approximation of MRF entropies}\label{MRFuniqueness}

The main measures we will study in this paper are Markov random fields (or MRFs) on sets of configurations on a graph $G$.

\begin{definition}
For any graph $G$ and finite alphabet $\A$, a measure $\mu$ on $\A^{V(G)}$ is called a {\rm $G$-Markov random field} (or $G$-MRF) if, for any finite $S \subset V(G)$, any $\eta \in \A^S$, any finite $T \subset V(G)$ s.t. $\partial(S,G) \subseteq T \subseteq V(G) \setminus S$, and any $\delta \in \A^T$ with $\mu(\delta) \neq 0$,
\[
\mu(\eta \ | \ \delta|_{\partial(S,G)}) = \mu(\eta \ | \ \delta).
\]
\end{definition}
Informally, $\mu$ is a $G$-MRF if, for any finite $S \subset V(G)$, the sites in $S$ and the sites in $V(G) \setminus (S \cup \partial(S,G))$ are $\mu$-conditionally independent given the sites on $\partial(S,G)$.

We will sometimes refer to a $G$-MRF simply as an MRF when $G$ is clear from context. We note that our definition of MRF differs slightly from the usual one, where the right-hand side would involve conditioning on an entire configuration on \newline $V(G) \setminus S$ a.e. rather than arbitrarily large finite subconfigurations of it. However, the definitions are equivalent and the finite approach leads to simpler calculations and proofs.


\begin{definition}
For any graph $G$ and finite alphabet $\A$, a {\rm $G$-specification} $\Lambda$ is defined by a set of finitely supported probability measures
\[
\{\Lambda^{\delta}(\cdot) \ | \ S \subset V(G), |S| < \infty, \delta \in \A^{\partial(S,G)}\},
\]
where, for each $\Lambda^{\delta}(\cdot)$, $\cdot$ ranges over all configurations in $\A^S$.
\end{definition}

Again we will sometimes refer to a $G$-specification simply as a specification. We say that a $\mathbb{Z}^d$-specification is \textit{translation-invariant} if $\Lambda^{\sigma_{v} \delta} = \sigma_v \Lambda^{\delta}$ for all $\delta$ and $v \in \zz^d$.

\begin{definition}
For any graph $G$, $\mu$ a measure on $\A^{V(G)}$, any finite set $S \subset V(G)$, and any $\delta \in \A^{\partial(S,G)}$ with $\mu(\delta) > 0$, denote by $\mu^{\delta}$ the measure on $\A^S$ defined by $\mu^{\delta}(u) = \mu(u \ | \ \delta)$.
\end{definition}

\begin{definition}
For any graph $G$ and finite alphabet $\A$, and $G$-specification $\Lambda$, a $G$-MRF $\mu$ {\rm is associated to} $\Lambda$ if $\mu^{\delta} = \Lambda^{\delta}$ for all finite $S \subseteq V(G)$ and $\delta \in \A^{\partial(S,G)}$ with $\mu(\delta) > 0$.
\end{definition}

Note that in checking whether an MRF $\mu$ is associated to a
specification $\Lambda$, many of the $\Lambda^{\delta}$ are totally
irrelevant; namely those which correspond to $\delta$ which have
zero $\mu$-measure.

We say that a $G$-specification $\Lambda$ is \textit{valid} if there
is at least one $G$-MRF associated to it. If there is exactly one
such MRF, we denote it by $\mu(\Lambda)$.

Often a specification is required to satisfy a consistency
condition; see [\cite{Georgii}, Definition 1.23].  This condition is
important for results that assert the existence of an MRF associated
to a given specification: the existence of an MRF $\mu$ forces
certain consistencies of specifications on the support of $\mu$.
However, in our work, we do not need to require consistency:
whenever we need existence, we will either assume it (i.e., that the
specification is valid) or assume a condition that guarantees it
(for instance, in Proposition~\ref{stripvalidity}). And the
consistency condition is not needed in uniqueness results such as
Theorem~\ref{uniqueness}.

To obtain good approximations of MRF entropies, we will use a condition from \cite{vdBM} on $G$-specifications, which was used there to prove uniqueness of the associated MRF. We first need some definitions.

\begin{definition}
For any finite set $S$ and two measures $\mu$ and $\nu$ on $S$, the {\rm variational distance} between $\mu$ and $\nu$ is
\[
d(\mu, \nu) = \frac{1}{2} \sum_{s \in S} |\mu(s) - \nu(s)|.
\]
\end{definition}

We note that $d(\mu, \nu) = 1$ iff $\mu$ and $\nu$ have disjoint supports.

\begin{definition}
For any graph $G$, finite alphabet $\A$, $g \in V(G)$, and valid $G$-specification $\Lambda$, define
\[
q_g(\Lambda) := \max_{\delta, \delta' \in \A^{N^G_g}} d(\Lambda^{\delta}, \Lambda^{\delta'})
\]
and $q(\Lambda) := \sup_{g \in V(G)} q_g(\Lambda)$.
\end{definition}

\begin{definition}
For any finite alphabet $\A$, graph $G$, and probability distribution $\lambda$ on $\A$, $P_{\lambda}$ represents the Bernoulli (i.i.d.) measure on $\A^{V(G)}$ whose distribution on each site is $\lambda$.
\end{definition}

\begin{definition}
For any graph $G$, the {\rm critical probability for site percolation on $G$}, denoted by $p_c(G)$, is defined as the supremum of $q \in [0,1]$ for which, given the alphabet $\{0,1\}$ and the graph $G$, the $P_{(1-q,q)}$-probability that there is an infinite connected subgraph of $G$ with $1$s at every site is zero.
\end{definition}

We point out that percolation theory is an extremely rich area of mathematics, which we give short shrift to here. For more information, see \cite{grimmett}.

When the graph $G$ is omitted, we understand $p_c$ to denote $p_c(\mathbb{Z}^2)$. Simulations suggest that $p_c \approx 0.593$, but the best known lower bound is $p_c > .556$, proved by van den Berg and Ermakov. (\cite{vdBE})

\begin{theorem}\label{uniqueness} {\rm (\cite{vdBM}, Corollary 2)} If $\Lambda$ is a valid $G$-specification and $q(\Lambda) < p_c(G)$, then there is a unique $G$-MRF associated to $\Lambda$.
\end{theorem}

Theorem~\ref{uniqueness} is, roughly speaking, proved by showing that for a $G$-specification $\Lambda$ with $q(\Lambda) < p_c(G)$, boundary conditions on large sets (such as rectangles in $\mathbb{Z}^2$) exert very little influence on sites near the center. It will be necessary for us to quantify exactly how this influence decays, and so we will use the methods of \cite{vdBM} to prove some finitistic results.

\begin{theorem}\label{vdBM} For any valid $\mathbb{Z}^2$-specification $\Lambda$ with $q(\Lambda) < p_c$ (with unique associated MRF $\mu = \mu(\Lambda)$), there exist $K, L > 0$ such that for any nonempty finite set $S \subset \mathbb{Z}^2$, for any rectangle $R \supset S$, and for any configurations $\delta$ and $\delta'$ on $\partial R$ with positive $\mu(\Lambda)$-probability, there exists $\lambda \in C(\mu^{\delta}|_S, \mu^{\delta'}|_S)$ such that for any $s \in S$, $\lambda(\{(x,y) \ : \ x(s) \neq y(s)\}) < Ke^{-Ld}$, where $d$ is the distance between $S$ and the set of $t \in \partial R$ for which $\delta(t) \neq \delta'(t)$. (We take $d = \infty$ if the latter set is empty.)
\end{theorem}

\noindent
\textit{Proof.} Given any such $\Lambda$, $S$, $R$, $\delta$, and $\delta'$, Theorem 1 from \cite{vdBM} proves the existence of $\lambda' \in C(\mu^{\delta},\mu^{\delta'})$ with the following two properties: (Following \cite{vdBM}, we will sometimes think of $\lambda'$ as a measure on $\A^{R \cup \partial R} \times \A^{R \cup \partial R}$, where pairs $(u,v)$ in the support of $\lambda'$ are thought of as equivalent to $(u \delta, v \delta')$.)\\

(i) Define the map $\phi$ from $(\A^{R \cup \partial R}) \times (\A^{R \cup \partial R})$ to $\{0,1\}^{R \cup \partial R}$ by $(\phi(x,y))(v) = 1$ if and only if $x(v) \neq y(v)$. Then the measure $\phi\lambda'$ on $\{0,1\}^{R \cup \partial R}$ is stochastically dominated by $P_{(1-q,q)}$, where $q = q(\Lambda)$. (We do not define stochastic dominance in general, but can give a simple definition which suffices for our setup. Given a finite set $S$ and measures $\mu$,$\nu$ on $\{0,1\}^S$, $\mu$ is stochastically dominated by $\nu$ if for any set $C$ of configurations which is closed under changing $0$s to $1$s, $\mu(C) \leq \nu(C)$.)

(ii) For a set of $(x,y) \in (\A^{R \cup \partial R}) \times (\A^{R \cup \partial R})$ with $\lambda'$-probability $1$, and for any $v \in R$, $x(v) \neq y(v)$ if and only if there is a path $P$ of sites in $\mathbb{Z}^2$ from $v$ to $\partial R$ for which $x(p) \neq y(p)$ for all $p \in P$.\\

Note that for any fixed $v \in R$, this means that
\begin{multline}\notag
\lambda'(\{(x,y) \ : \ x(v) \neq y(v)\})\\
= \lambda'(\{(x,y) \ : \ \textrm{there is a path } P \textrm{ from } v \textrm{ to } \partial R \textrm{ such that } x(p) \neq y(p) \ \forall p \in P\})\\
= (\phi \lambda')(\textrm{there is a path } P \textrm{ of } $1$\textrm{s from } v \textrm{ to } \partial R\})\\
\leq P_{(1-q,q)}(\textrm{there is a path of } $1$\textrm{s from } v \textrm{ to } \partial R).
\end{multline}

A classical theorem proved by Menshikov (\cite{Me}) and Aizenmann and Barsky (\cite{AB}) shows that for any
$q < p_c$, there exist $K = K(q)$ and $L = L(q)$ so that for any $n$, $P_{(1-q,q)}(\textrm{there is a path of } $1$\textrm{s from } 0 \textrm{ to } \partial [-n,n]^2) < Ke^{-Ln}$. This clearly implies that $P_{(1-q,q)}(\textrm{there is a path of } $1$\textrm{s from } v \textrm{ to } \partial R) < Ke^{-Ld_v}$, where $d_v$ is the distance from $v$ to the set of sites in $\partial R$ at which $\delta$ and $\delta'$ disagree.

Therefore, if we define $\lambda = \lambda'|_{(\A^S \times \A^S)}$, then for any $s \in S$, $\lambda(\{(x,y) \ : \ x(s) \neq y(s)\}) = \lambda'(\{(x,y) \ : \ x(s) \neq y(s)\}) \leq Ke^{-Ld_s} \leq Ke^{-Ld}$, where $d$ is just the minimum value of $d_s$ for $s \in S$.

\begin{flushright}
$\blacksquare$\\
\end{flushright}

We will prove a slightly more general version of Theorem~\ref{vdBM} for $\delta$ and $\delta'$ on the boundaries of possibly different rectangles.

\begin{theorem}\label{gen-vdBM} For any valid $\mathbb{Z}^2$-specification $\Lambda$ with $q(\Lambda) < p_c$ (with unique associated MRF $\mu = \mu(\Lambda)$), there exist $K, L > 0$ such that for any finite set $S \subset \mathbb{Z}^2$, for any rectangles $R' \supset R \supset S$, and for any configurations $\delta$ and $\delta'$ on $\partial R$ and $\partial R'$ respectively with positive $\mu(\Lambda)$-probability, there exists $\lambda \in C(\mu^{\delta}|_S, \mu^{\delta'}|_S)$ such that for any $s \in S$, $\lambda(\{(x,y) \ : \ x(s) \neq y(s)\}) < Ke^{-Ld}$, where $d$ is the distance between $S$ and the set of $t \in \partial R$ for which either $t \notin \partial R'$ or $t \in \partial R'$ and $\delta(t) \neq \delta'(t)$.
\end{theorem}

\noindent
\textit{Proof.} We will be using Theorem~\ref{vdBM}. First, let's note that $\mu^{\delta'}|_{R}$ can be written as a weighted average of the measures $\mu^{\eta}$, where $\eta$ ranges over all configurations on $\partial R$ which agree with $\delta'$ on $\partial R \cap \partial R'$.
This means that in particular, $\mu^{\delta'}|_S = \sum_{\eta} \alpha_{\eta} (\mu^{\eta}|_S)$ for some nonnegative numbers $\alpha_{\eta}$ summing to $1$. By Theorem~\ref{vdBM} there exist $K, L > 0$ so that for any $\eta$ there exists $\lambda_{\eta} \in C(\mu^{\delta}|_S, \mu^{\eta}|_S)$ with the property that for any $s \in S$, $\lambda_{\eta}(\{(x,y) \ : \ x(s) \neq y(s)\}) < Ke^{-Ld_{\eta}}$, where $d_{\eta}$ is the distance from any $s \in S$ to the set of $t \in \partial R$ for which $\delta(t) \neq \eta(t)$. Take $\lambda = \sum_{\eta} \alpha_{\eta} \lambda_{\eta}$. It is clear that $\lambda \in C(\mu^{\delta}|_S, \mu^{\delta'}|_S)$. Also, clearly

\[
\lambda(\{(x,y) \ : \ x(s) \neq y(s)\}) = \sum_{\eta} \alpha_{\eta} \lambda_{\eta}(\{(x,y) \ : \ x(s) \neq y(s)\}) \leq \sum_{\eta} \alpha_{\eta} Ke^{-Ld_{\eta}}.
\]

Note that for any $\eta$ agreeing with $\delta'$ on $\partial R \cap \partial R'$, $d_{\eta} \geq d$ as defined in the theorem. Therefore, clearly $\lambda(\{(x,y) \ : \ x(s) \neq y(s)\}) < Ke^{-Ld}$, and we are done.

\begin{flushright}
$\blacksquare$\\
\end{flushright}

Given a valid $\mathbb{Z}^2$-specification $\Lambda$ with $q(\Lambda) < p_c$, we wish to define some specifications (and associated MRFs) on the maximal subgraphs $H_{m,n}$ of $\mathbb{Z}^2$ defined by $V(H_{m,n}) = \mathbb{Z} \times [m,n]$, yielding measures on sets of configurations on biinfinite horizontal strips.


Of course $\Lambda$ itself does not contain enough information to define an \newline $H_{m,n}$-specification; there exist sets of the form $\partial(S,H_{m,n})$ for some finite set $S \subset H_{m,n}$, but which are not expressible as $\partial T$ for any finite set $T \subset \mathbb{Z}^2$. (For instance, picture the `three-sided' boundary of a rectangle which includes part of the top row in $H_{m,n}$.) We therefore supplement $\Lambda$ with boundary conditions as follows. Suppose that $t,b \in \A^{\mathbb{Z}}$, and we will define a $H_{m,n}$-specification $\Lambda_{m,n,t,b}$.

For any finite $S \subset \mathbb{Z} \times [m,n]$, $\eta \in \A^S$, and $\delta \in \A^{\partial(S,H_{m,n})}$, define
\begin{equation}\label{mntbdefn}
\Lambda_{m,n,t,b}^{\delta}(\eta) = \Lambda^{\xi}(\eta),
\end{equation}
where $\xi = (t|_{\partial S \cap (\mathbb{Z} \times \{n+1\})}) (b|_{\partial S \cap (\mathbb{Z} \times \{m-1\})}) \delta$, the concatenation of $t|_{\partial S \cap (\mathbb{Z} \times \{n+1\})}$,
$b|_{\partial S \cap (\mathbb{Z} \times \{m-1\})}$, and $\delta$. In other words, a configuration on $\partial(S,H_{m,n})$ is supplemented by symbols from $t$ and $b$ above and below $H_{m,n}$, if necessary, to extend it to $\partial S$.
The following gives a
sufficient condition on $m,n,t,b$ for the validity of $\Lambda_{m,n,t,b}$.

\begin{definition}\label{compatible}
For a $\zz^2$-specification $\Lambda$ with alphabet $\A$,
integers $m < n$ and $t, b \in \A^{\mathbb{Z}}$, we say that
$m,n,t,b$ {\rm is compatible with} $\Lambda$ if there exists an
MRF $\mu$ associated to $\Lambda$ such that for all sufficiently large $k$, there exists $\delta_k \in \A^{\partial ([-k,k] \times [m,n])}$ with positive $\mu$-measure whose top row is $t|_{[-k,k]}$ and whose bottom row is $b|_{[-k,k]}$.
\end{definition}

\begin{proposition}\label{stripvalidity}
For a valid $\zz^2$-specification $\Lambda$ with alphabet $\A$,
integers $m < n$ and $t, b \in \A^{\mathbb{Z}}$, if
$m,n,t,b$ is compatible with $\Lambda$, then
$\Lambda_{m,n,t,b}$ is a valid $H_{m,n}$-specification.
\end{proposition}

\begin{proof}
Let $\mu$ be an associated MRF and for all sufficiently large $k$, $\delta_k$ as described
in Definition~\ref{compatible}.
 Define the measures $\mu^{\delta_k}$. We wish to take a weak limit of a subsequence of these measures, so they must be extended to measures on all of $\A^{\zz \times [m,n]}$: choose any $a \in \A$, and for each $k$ extend each configuration in the support of $\mu^{\delta_k}$ to all of
 $\zz \times [m,n]$ by filling all unoccupied sites with $a$'s.

By definition of $\Lambda_{m,n,t,b}$ and the fact that $\mu$ is an MRF associated to $\Lambda$, any weak limit of a subsequence of the measures $\mu^{\delta_k}$ is an MRF associated to the specification $\Lambda_{m,n,t,b}$, and so $\Lambda_{m,n,t,b}$ is valid.
\end{proof}

While the compatibility condition may be difficult to check in general, it is checkable
for certain special kinds of MRFs introduced in later sections.

We can also give a sufficient condition for uniqueness of an $H_{m,n}$-MRF associated to $\Lambda_{m,n,t,b}$. In fact, it requires a less restrictive bound on $q(\Lambda)$, which will be useful for some later discussions.

\begin{proposition}\label{stripuniqueness}
For any integers $m<n$ and any valid $H_{m,n}$-specification $\Lambda_{m,n,t,b}$ induced by a $\mathbb{Z}^2$-specification $\Lambda$ with $q(\Lambda) < 1$ and boundary conditions $t$ and $b$, $\Lambda_{m,n,t,b}$ has a unique associated MRF.
\end{proposition}

\begin{proof}
Consider any such $\Lambda$, $m$, $n$, $t$, and $b$ for which $\Lambda_{m,n,t,b}$ is a valid specification. For any $i \in V(H_{m,n}) = \zz \times [m,n]$, by the definition of $\Lambda_{m,n,t,b}$, $q_i(\Lambda_{m,n,t,b})$ and $q_i(\Lambda)$ are both maxima of $d(\Lambda^{\delta}, \Lambda^{\delta'})$ over sets of pairs $\delta,\delta' \in \A^{N_i}$. However, for $q_i(\Lambda)$, one maximizes over all such pairs, and for $q_i(\Lambda_{m,n,t,b})$, one may be maximizing over a smaller set. (For instance, if $i$ is part of the top row of $H_{m,n}$, then one only considers configurations on the neighbors where the neighbor above $i$ is equal to $t(i+(0,1))$) Therefore, $q_i(\Lambda_{m,n,t,b}) \leq q_i(\Lambda)$. Since $i$ was arbitrary, $q(\Lambda_{m,n,t,b}) \leq q(\Lambda)$.

It now suffices to show that $p_c(H_{m,n}) = 1$; the proposition then follows from Theorem~\ref{uniqueness}. For any $p < 1$, if sites of $H_{m,n}$ are independently taken to be open with
probability $p$ and closed with probability $1-p$, then the probability than an entire column $\{i\} \times [m,n]$ is closed is $(1-p)^{n-m+1} > 0$. This means that with $P_{(1-p,p)}$-probability $1$, there exist closed columns arbitrarily far to the left and right, meaning that there are no infinite open connected clusters. Therefore, since $p<1$ was arbitrary, $p_c(H_{m,n}) = 1$ and we are done.
\end{proof}

Clearly, when the hypotheses of Proposition~\ref{stripuniqueness} are satisfied, $\mu(\Lambda_{m,n,t,b})$ can be thought of as a measure on the one-dimensional full shift $(\A^{[m,n]})^{\zz}$, and we will interpret $\mu(\Lambda_{m,n,t,b})$ in this way for further discussions about $\dbar$ distance and entropy. It should always be clear from context which viewpoint is being used.

\begin{proposition}\label{weaklimit}
For any valid $\mathbb{Z}^2$-specification $\Lambda$ with $q(\Lambda) < p_c$ (with unique associated MRF $\mu = \mu(\Lambda)$) and any $t,b \in \A^{\zz}$ for which there exists $N$ such that $\mu(\Lambda_{-n,n,t,b})$ exists for all $n>N$, $\mu(\Lambda_{-n,n,t,b})$ approaches $\mu$ weakly as $n \rightarrow \infty$.
\end{proposition}

\begin{proof}
By definition of $\Lambda_{-n,n,t,b}$, any weak limit of a subsequence of $\mu(\Lambda_{-n,n,t,b})$ is clearly a $\mathbb{Z}^2$-MRF associated to $\Lambda$. (As before, we need to extend each $\mu(\Lambda_{-n,n,t,b})$ to a measure on all of $\A^{\zz^2}$; we do this by choosing any $a \in \A$ and extending each configuration in the support of $\mu(\Lambda_{-n,n,t,b})$ to all of $\zz^2$ by filling the unoccupied sites with $a$.) However, the only such MRF is $\mu$.
\end{proof}

We will now use Theorem~\ref{gen-vdBM} to derive couplings of marginalizations of $\mu(\Lambda_{1,n,t,b})$ and $\mu(\Lambda_{1,n+1,t,b})$ to substrips which imply their closeness in the $\dbar$ metric.

\begin{theorem}\label{dbarclose}
For any valid $\mathbb{Z}^2$-specification $\Lambda$ with $q(\Lambda) < p_c$ (with unique associated MRF $\mu = \mu(\Lambda)$), there exist $K, L > 0$ such that for any $n$ and any $t,b \in \A^{\mathbb{Z}}$ such that
$1,n,t,b$ and $1,n+1,t,b$ are compatible with $\Lambda$, and for any $1 \leq i < i' \leq n$,
\[
\dbar\big(\mu(\Lambda_{1,n,t,b})|_{\zz \times [i,i'-1]}, \mu(\Lambda_{1,n+1,t,b})|_{\zz \times [i,i'-1]}\big) \leq (i'-i)Ke^{-L(n-i')} \textrm{ and}
\]
\[
\dbar\big(\mu(\Lambda_{1,n,t,b})|_{\zz \times [i,i'-1]}, \mu(\Lambda_{1,n+1,t,b})|_{\zz \times [i+1,i']}\big) \leq (i'-i)Ke^{-Li}.
\]
\end{theorem}

\begin{proof}
We begin with the first inequality. For $\Lambda$ as in the theorem, take the $K$ and $L$ guaranteed by Theorem~\ref{gen-vdBM}. Fix $t$, $b$, $n$, and $i<i'$. For every sufficiently large $k$, take $\delta_k$ and $\delta'_k$ configurations on $\partial([-k,k] \times [1,n])$ and $\partial([-k,k] \times [1,n+1])$ respectively with positive $\mu$-measure which are both equal to $t|_{[-k,k]}$ on the top and $b|_{[-k,k]}$ on the bottom. Define $S_{i,i',k} = [-k,k] \times [i,i'-1]$. By Theorem~\ref{gen-vdBM}, for any $k > n-i'$ and $j < k - (n-i')$, there exists $\lambda_{i,i',j,k} \in C(\mu^{\delta_k}|_{S_{i,i',j}}, \mu^{\delta'_k}|_{S_{i,i',j}})$ for which the $\lambda_{i,i',j,k}$-probability of disagreement at any site in $S_{i,i',j}$ is less than $Ke^{-L(n-i')}$. (This is because $\delta_k$ and $\delta'_k$ agree on their bottom row, and the distance from any site in $S_{i,i',j}$ to any other site in $\partial([-k,k] \times [1,n]) \cup \partial([-k,k] \times [1,n+1])$ is at least $n-i'$.)

Now, for any fixed $i < i'$ and $j$, define $\lambda_{i,i',j}$ to be any weak limit of a subsequence of the couplings $\lambda_{i,i',j,k}$ as $k \rightarrow \infty$. Note that any weak subsequence of $\mu^{\delta_k}$ approaches an MRF on $H_{1,n,t,b}$ associated to the specification $\Lambda_{1,n,t,b}$, and since $\mu(\Lambda_{1,n,t,b})$ is the unique such MRF, the sequence $\mu^{\delta_k}$ itself must weakly approach $\mu(\Lambda_{1,n,t,b})$. Similarly, $\mu^{\delta'_k}$ weakly approaches $\mu(\Lambda_{1,n+1,t,b})$. Therefore, $\lambda_{i,i',j} \in C(\mu(\Lambda_{1,n,t,b})|_{S_{i,i',j}}, \mu(\Lambda_{1,n+1,t,b})|_{S_{i,i',j}})$. Also, it is clear that if we think of a configuration on $S_{i,i',j}$ as a $(2j+1)$-letter word on the alphabet of words on columns of height $i'-i$, then $E_{\lambda_{i,i',j}}(d_{2j+1}(x|_{[-j,j]}, y|_{[-j,j]})) < (i'-i)Ke^{-L(n-i')}$. Recalling the definition of $\dbar$, we see that then clearly $\dbar(\mu(\Lambda_{1,n,t,b})|_{\zz \times [i,i'-1]}, \mu(\Lambda_{1,n+1,t,b})|_{\zz \times [i,i'-1]}) \leq (i'-i)Ke^{-L(n-i')}$.

To prove the second inequality, change the proof above by defining $\delta_k$ on \newline $\partial([-k,k] \times [2,n+1])$ instead. Then $\delta_k$ and $\delta'_k$ will agree on their top rows, and the rest of the proof goes through mostly unchanged. The distances from $\zz \times [i,i'-1]$ and $\zz \times [i+1,i']$ to the bottom rows of $H_{1,n}$ and $H_{1,n+1}$ respectively are at least $i$, which is why $n-i'$ is replaced by $i$.
\end{proof}

We will for now restrict our attention to translation-invariant $\mathbb{Z}^2$-specifications $\Lambda$ with $q(\Lambda) < p_c$
and constant boundary conditions $t,b \in \A^{\mathbb{Z}}$, which means that the measures $\mu(\Lambda_{m,n,t,b})$ (when they exist)
 will be translation-invariant
as one-dimensional measures. (Otherwise, their horizontal shifts
would also be MRFs associated to $\Lambda_{m,n,t,b}$, contradicting
uniqueness of $\mu(\Lambda_{m,n,t,b})$.) We can then discuss the
measure-theoretic entropies $h(\mu(\Lambda_{1,n,t,b}))$ and
$h(\mu(\Lambda_{1,n+1,t,b}))$.

We will decompose these into conditional measure-theoretic
entropies, and then use Theorem~\ref{dbarclose}
and H\"{o}lder continuity of entropy (with respect to $\dbar$)
to show that many of these entropies are exponentially close,
finally showing that $h(\mu(\Lambda_{1,n+1,t,b})) -
h(\mu(\Lambda_{1,n,t,b}))$ is exponentially close to $h(\mu)$. We
first need some notation for special conditional measure-theoretic
entropies.

For any $a$, define $R_a = \zz \times \{a\}$. For any $m < n$ and any interval $B \subseteq [m,n]$, we partition $\A^{[m,n]}$ by the letters appearing on $B$, and call this partition $\xi_{B}$. Then, for any translation-invariant measure $\mu$ on $(\A^{[m,n]})^{\zz}$ and disjoint adjacent intervals $B,C \subseteq [m,n]$, we make the notations

\[
h_{\mu}\Big(\bigcup_{b \in B} R_b\Big) := h(\phi_{\xi_B}(\mu)) \textrm{ and}
\]
\[
h_{\mu}\Big(\bigcup_{c \in C} R_c \ | \ \bigcup_{b \in B} R_b\Big) := h(\phi_{\xi_{B \cup C}}(\mu) \ | \ \xi_B).
\]

Also, for any translation-invariant measure $\mu$ on $\A^{\zz^2}$, $h_{\mu}\Big(\bigcup_{b \in B} R_b\Big)$ will be understood to mean $h_{\mu|_{\left(\bigcup_{d \in D} R_d\right)}}\Big(\bigcup_{b \in B} R_b\Big)$ for any $D \supseteq B$; in other words, this expression is given meaning by marginalizing $\mu$ to any substrip containing the rows whose entropy is to be computed. It does not matter which $D$ is used, since clearly this quantity depends only on the restriction of $\mu$ to $\bigcup_{b \in B} R_b$. We will interpret $h_{\mu}\Big(\bigcup_{c \in C} R_c \ | \ \bigcup_{b \in B} R_b\Big)$ in an analogous fashion.

The following is just a consequence of Proposition~\ref{P2} for this new notation.
\begin{proposition}\label{P3}
For any $m<n$, $\mu$ a translation-invariant measure on $(\A^{[m,n]})^{\zz}$ or $\A^{\zz^2}$, and $B,C$ adjacent subintervals of $[m,n]$,
\[
h_{\mu}\Big(\bigcup_{a \in B \cup C} R_a\Big) = h_{\mu}\Big(\bigcup_{b \in B} R_b\Big) + h_{\mu}\Big(\bigcup_{c \in C} R_c \ | \ \bigcup_{b \in B} R_b\Big).
\]
\end{proposition}

The following theorem can both be thought of as an extension of the Markov property in one dimension and a generalization of Theorem 13 from \cite{Pa}.

\begin{theorem}\label{MRFstripentropy}
Let $\Lambda$ be a
translation-invariant $\mathbb{Z}^2$-specification. Let
$n$ be a positive integer, $t,b \in \A^{\mathbb{Z}}$ be constant sequences and $\mu$ be an MRF associated to  $\Lambda_{1,n,t,b}$. Then for any integers $i$ and $k$ with $1 \leq k < i \leq n$,
$$
h_\mu\left(R_i \ | \ \bigcup_{j=k}^{i-1} R_j\right) = h_\mu(R_i \ | \ R_{i-1}).
$$
\end{theorem}

\begin{proof}
We will prove the theorem for $k = 1$, and this suffices to prove the theorem for all $k < i$, since for any $1 \leq k < i$, conditioning on $\xi_{[k,i-1]}$ is an intermediate partition between $\xi_{[1,i-1]}$ and $\xi_{\{i-1\}}$. If the conditional entropies resulting from these two partitions are equal, then clearly any intermediate partition gives the same value. Fix any $\Lambda$, $n$, $t$, $b$, and $i$ as in the statement of the theorem. For simplicity, we write $\mu = \mu(\Lambda_{1,n,t,b})$. We can write
\[
h_{\mu}\Big(R_i \ | \ \bigcup_{j=1}^{i-1} R_j\Big) = \lim_{k \rightarrow \infty} (1/k) S_k, \textrm{ where }
\]
\[
S_k := \sum_{\substack{w \in \A^{[-k,k] \times [1,i-1]},\\ v \in \A^{[-k,k] \times \{i\}}}} \mu(w \cap v) \log \Big(\frac{\mu(w)}{\mu(w \cap v)}\Big)
\]
\begin{equation}\label{eq1}
= \bigg(\sum_{w \in \A^{[-k,k] \times [1,i-1]}} \mu(w) \log \mu(w)\bigg)
\end{equation}
\begin{equation}\label{eq2}
- \bigg(\sum_{\substack{w \in \A^{[-k,k] \times [1,i-1]},\\ v \in \A^{[-k,k] \times \{i\}}}} \mu(w \cap v) \log \mu(w \cap v)\bigg).
\end{equation}
(As in the definition of measure-theoretic entropy, in each sum we omit terms coming from configurations of $\mu$-measure zero.) We also define
\[
S^*_k := \sum_{\substack{w \in \A^{[-k,k] \times [1,i-1]},\\ v \in \A^{[-k,k] \times \{i\}},\\ L \in \A^{\{-k-1\} \times [1,i-1]},\\ R \in \A^{\{k+1\} \times [1,i-1]}}} \mu(w \cap v \cap L \cap R) \log\Big(\frac{\mu(w \cap L \cap R)}{\mu(w \cap v \cap L \cap R)}\Big) =
\]
\begin{equation}\label{eq3}
\bigg(\sum_{\substack{w \in \A^{[-k,k] \times [1,i-1]},\\ L \in \A^{\{-k-1\} \times [1,i-1]},\\ R \in \A^{\{k+1\} \times [1,i-1]}}} \mu(w \cap L \cap R) \log\mu(w \cap L \cap R)\bigg) -
\end{equation}
\begin{equation}\label{eq4}
\bigg(\sum_{\substack{w \in \A^{[-k,k] \times [1,i-1]},\\ v \in \A^{[-k,k] \times \{i\}},\\ L \in \A^{\{-k-1\} \times [1,i-1]},\\ R \in \A^{\{k+1\} \times [1,i-1]}}} \mu(w \cap v \cap L \cap R) \log\mu(w \cap v \cap L \cap R)\bigg).
\end{equation}

We claim that $|S_k - S^*_k| \leq 4(i-1) \log |\A|$. To see this, we first compare (\ref{eq1}) and (\ref{eq3}). For any fixed $w$, consider the term $\mu(w) \log \mu(w)$ from (\ref{eq1}). Compare this to the corresponding terms from (\ref{eq3}), i.e. $\sum_{L,R} \mu(w \cap L \cap R) \log \mu(w \cap L \cap R)$. We make the simple observation that for any set of $k$ nonnegative reals $\{\alpha_i\}_{i=1}^{k}$ summing to $\alpha$, $-\sum_{i=1}^k \alpha_i \log(\alpha_i)$ is at least $-\alpha \log \alpha$, and at most $-\alpha \log (\frac{\alpha}{k})$ (achieved when all $\alpha_i$ are equal). Therefore,
\[
\Big|\mu(w) \log \mu(w) - \sum_{L,R} \mu(w \cap L \cap R) \log \mu(w \cap L \cap R)\Big| \leq \mu(w) \log(|\A|^{2(i-1)})
\]

\noindent
since the number of different pairs $L,R$ is at most $|\A|^{2(i-1)}$. By summing this over all choices of $w$ in (\ref{eq1}) and (\ref{eq3}), we see that the difference between (\ref{eq1}) and (\ref{eq3}) has absolute value at most $\sum_{w} \mu(w) \log(|\A|^{2(i-1)}) = 2(i-1) \log |\A|$. An analogous argument may be made for (\ref{eq2}) and (\ref{eq4}), and so $|S_k - S^*_k| \leq 4(i-1) \log |\A|$ for all $k$.

We can similarly write $h_{\mu}(R_i \ | \ R_{i-1}) = \lim_{k \rightarrow \infty} (1/k) T_k$, where $T_k$ is defined exactly as $S_k$, but where $w \in \A^{[-k,k] \times [1,i-1]}$ is replaced in the summation by $w' \in \A^{[-k,k] \times \{i-1\}}$. If we define $T^*_k$ exactly as $S^*_k$, again using $w'$ instead of $w$, then a trivially similar proof to the above shows that $|T_k - T^*_k| \leq 4(i-1) \log |\A|$ for all $k$. If we can now prove that $S^*_k = T^*_k$ for all $k$, then $|S_k - T_k| \leq |S_k - S^*_k| + |T_k - T^*_k| + |S^*_k - T^*_k| \leq 8(i-1) \log |\A|$ for every $k$, which clearly shows that $h_{\mu}\Big(R_i \ | \ \bigcup_{j=1}^{i-1} R_j\Big) = \lim_{k \rightarrow \infty} (1/k) S_k$ and $h_{\mu}(R_i \ | \ R_{i-1}) = \lim_{k \rightarrow \infty} (1/k) T_k$ are equal.

We claim that for any $L \in \A^{\{-k-1\} \times [1,i-1]}$, $R \in \A^{\{k+1\} \times [1,i-1]}$, $v \in \A^{[-k,k] \times \{i\}}$, and $w \in \A^{[-k,k] \times [1,i-1]}$, if we define $w' = w|_{[-k,k] \times \{i-1\}}$, then
\[
\frac{\mu(w \cap L \cap R)}{\mu(w \cap v \cap L \cap R)} = \frac{\mu(w' \cap L \cap R)}{\mu(w' \cap v \cap L \cap R)}.
\]

To see this, we define $w'' = w|_{[-k,k] \times [1,i-2]}$ and note that since $\mu$ is an $H_{m,n}$-MRF,
\[
\frac{\mu(w'' \cap w' \cap L \cap R)}{\mu(w' \cap L \cap R)} = \frac{\mu(w'' \cap w' \cap v \cap L \cap R)}{\mu(w' \cap v \cap L \cap R)}.
\]

From this, it is clear that
\[
\frac{\mu(w' \cap L \cap R)}{\mu(w' \cap v \cap L \cap R)} = \frac{\mu(w'' \cap w' \cap L \cap R)}{\mu(w'' \cap w' \cap v \cap L \cap R)},
\]
which is clearly equal to $\frac{\mu(w \cap L \cap R)}{\mu(w \cap v \cap L \cap R)}$. But then for any fixed $w'$, all terms in $S^*_k$ corresponding to $w$ with the top row $w'$ can be collapsed, which quickly yields $S^*_k = T^*_k$, completing the proof.

\end{proof}

\begin{corollary}\label{MRFentropy}
If $\Lambda$ is a translation-invariant $\mathbb{Z}^2$-specification
with $q(\Lambda) < p_c$ (with unique associated MRF $\mu = \mu(\Lambda)$),
and $t$ and $b$ are constant sequences
such that  $-n,n,t,b$ is compatible with $\Lambda$ for sufficiently
large $n$, and $k$ is any positive integer, then
$$
h_{\mu}\left(R_0 \ | \ \bigcup_{j=-k}^{-1} R_j\right) = h_{\mu}(R_0 \ | \ R_{-1}).
$$
\end{corollary}

\begin{proof}
Fix any such $\mu$, $t$, $b$, and $N>0$ such that
 $-n,n,t,b$ is compatible with $\Lambda$
for $n>N$. Let $k$ be a positive integer. Then by
Theorem~\ref{dbarclose}, there exist $K,L>0$ so that
\[
\dbar(\mu(\Lambda_{-n,n,t,b})|_{\bigcup_{j=-(k-1)}^0 R_j}, \mu(\Lambda_{-(n+1),n+1,t,b})|_{\bigcup_{j=-(k-1)}^0 R_j}) < kKe^{-L(n-k)}
\]
\noindent for any $n>N$ and $k \leq n$, and so this sequence of
marginalizations is $\dbar$ Cauchy and approaches a $\dbar$ limit.
Since the measures $\mu(\Lambda_{-n,n,t,b})$ approach $\mu$ weakly
as $n \rightarrow \infty$ by Proposition~\ref{weaklimit}, it must be
the case that the $\dbar$ limit of
$\mu(\Lambda_{-n,n,t,b})|_{\bigcup_{j=-k}^0 R_j}$ is
$\mu|_{\bigcup_{j=-k}^0 R_j}$.
Thus, by Theorem~\ref{MRFstripentropy}, translation-invariance of
$\mu$, and continuity of entropy with respect to the $\dbar$ metric,
we conclude that for any positive integer $k$, $h_{\mu}(R_0 \ | \
\bigcup_{j=-k}^{-1} R_j) = h_{\mu}(R_0 \ | \ R_{-1})$.

\end{proof}

\begin{corollary}\label{MRFentropy2}
If $\Lambda$ is a translation-invariant $\mathbb{Z}^2$-specification with $q(\Lambda) < p_c$ (with unique associated MRF $\mu = \mu(\Lambda)$), and $t$ and $b$ are constant sequences
such that $-n,n,t,b$ is compatible with $\Lambda$ for sufficiently large $n$,
then
$h(\mu) = h_{\mu}(R_0 \ | \ R_{-1})$.
\end{corollary}

\begin{proof}
For any such $\mu$ and positive integer $k$, $h_{\mu}(\bigcup_{j=0}^k R_j) = h_{\mu}(R_0) + h_{\mu}(R_1 \ | \ R_0) + \ldots + h_{\mu}(R_k \ | \ \bigcup_{j=0}^{k-1} R_j) = h_{\mu}(R_0) + (k-1)h_{\mu}(R_0 \ | \ R_{-1})$ by translation-invariance of $\mu$, Proposition~\ref{P3}, and Corollary~\ref{MRFentropy}. But then by dividing by $k$ and letting $k \rightarrow \infty$, we see that $h(\mu) = h_{\mu}(R_0 \ | \ R_{-1})$.

\end{proof}

\begin{theorem}\label{maintheorem}
For any valid translation-invariant $\mathbb{Z}^2$-specification $\Lambda$ with $q(\Lambda) < p_c$ (with unique associated MRF $\mu = \mu(\Lambda)$) and any integer $N$ and constant sequences $t$ and $b$ such
that $1,n,t,b$ is compatible with $\Lambda$ when $n>N$, there exist $Q, R > 0$ such that $|h(\mu(\Lambda_{1,n+1,t,b})) - h(\mu(\Lambda_{1,n,t,b})) - h(\mu)| < Qe^{-Rn}$ for any $n>N$.
\end{theorem}

\begin{proof}
Proposition~\ref{P3} and Theorem~\ref{MRFstripentropy} imply that for $n>N$,
\[
h(\mu(\Lambda_{1,n+1,t,b})) = h_{\mu(\Lambda_{1,n+1,t,b})}(R_1) + \sum_{j=1}^{n} h_{\mu(\Lambda_{1,n+1,t,b})}(R_{j+1} \ | \ R_j) \textrm{ and}
\]
\[
h(\mu(\Lambda_{1,n,t,b})) = h_{\mu(\Lambda_{1,n,t,b})}(R_1) + \sum_{j=1}^{n-1} h_{\mu(\Lambda_{1,n,t,b})}(R_{j+1} \ | \ R_j).
\]

We can then write $h(\mu(\Lambda_{1,n+1,t,b})) - h(\mu(\Lambda_{1,n,t,b}))$ as
\begin{align}
& h_{\mu(\Lambda_{1,n+1,t,b})}(R_1) - h_{\mu(\Lambda_{1,n,t,b})}(R_1)\label{term1}\\
& + \sum_{j=1}^{\lfloor \frac{n}{2} \rfloor - 1} \Big(h_{\mu(\Lambda_{1,n+1,t,b})}(R_{j+1} \ | \ R_j) - h_{\mu(\Lambda_{1,n,t,b})}(R_{j+1} \ | \ R_j)\Big)\label{term2}\\
& + h_{\mu(\Lambda_{1,n+1,t,b})}(R_{\lfloor \frac{n}{2} \rfloor + 1} \ | \ R_{\lfloor \frac{n}{2} \rfloor})\label{term3}\\
& + \sum_{j=\lfloor \frac{n}{2} \rfloor + 1}^{n} \Big(h_{\mu(\Lambda_{1,n+1,t,b})}(R_{j+1} \ | \ R_j) - h_{\mu(\Lambda_{1,n,t,b})}(R_j \ | \ R_{j-1})\Big).\label{term4}
\end{align}

By Theorem~\ref{dbarclose} and H\"{o}lder continuity of entropy with
respect to $\dbar$, we see that (\ref{term1}), (\ref{term2}), and
(\ref{term4}) are exponentially small in $n$, i.e. there exist
constants $Q$ and $R$ independent of $n$ such that each has absolute
value smaller than $Qe^{-Rn}$. Theorem~\ref{dbarclose} also implies
that the sequence $\mu(\Lambda_{1,n+1,t,b})|_{R_{\lfloor \frac{n}{2}
\rfloor} \cup R_{\lfloor \frac{n}{2} \rfloor + 1}}$ approaches a
$\dbar$ limit with exponential rate in $n$, as $n \rightarrow
\infty$. Since $\sigma_{(0,-\lfloor \frac{n}{2} \rfloor -1)}
\mu(\Lambda_{1,n+1,t,b})$ approaches $\mu$ weakly by
Proposition~\ref{weaklimit}, this $\dbar$ limit must be
$\mu|_{R_{-1} \cup R_0}$. But then again by H\"{o}lder continuity of
entropy with respect to $\dbar$ and translation-invariance of $\mu$,
(\ref{term3}) approaches $h_{\mu}(R_0 \ | \ R_{-1})$ with
exponential rate as $n \rightarrow \infty$, which equals $h(\mu)$ by
Corollary~\ref{MRFentropy2}.
\end{proof}

This result can be generalized to periodic boundary conditions as well. If $\Lambda$ is a translation-invariant $\mathbb{Z}^2$-specification with $q(\Lambda) < p_c$ and $t$ and $b$ are periodic sequences such that $1,n,t,b$ is compatible with $\Lambda$, then $\mu(\Lambda_{1,n,t,b})$ exists by Propositions~\ref{stripvalidity} and \ref{stripuniqueness}, and $(\mu(\Lambda_{1,n,t,b}))^{[p]}$ (the $p$-higher power code of $\mu(\Lambda_{1,n,t,b})$) is translation-invariant. 

\begin{theorem}\label{maintheorem2}
For any valid translation-invariant $\mathbb{Z}^2$-specification $\Lambda$ with $q(\Lambda) < p_c$ (with unique associated MRF $\mu = \mu(\Lambda)$), and any integer $N$ and sequences $t$ and $b$ with period $p$ 
such that $1,n,t,b$ is compatible with $\Lambda$ when $n>N$, there exist $Q, R > 0$ such that $|(1/p)(h((\mu(\Lambda_{1,n+1,t,b}))^{[p]}) - h((\mu(\Lambda_{1,n,t,b}))^{[p]}) - h(\mu)| < Qe^{-Rn}$ for all $n>N$.
\end{theorem}

\begin{proof}
The proof is nearly identical to
that of Theorem~\ref{maintheorem}, and so we only highlight the
slight differences. Theorem~\ref{dbarclose} implies the exponential
$\dbar$ closeness of relevant marginalizations of
$\mu(\Lambda_{1,n,t,b})$ to substrips just as before. Then, since
passing to $(\mu(\Lambda_{1,n,t,b}))^{[p]}$ multiplies the relevant
$\dbar$ distances by at most $p$, we still have the necessary
exponential $\dbar$ closeness of marginalizations for these recoded
strip measures. Since the $(\mu(\Lambda_{1,n,t,b}))^{[p]}$ are
translation-invariant, the same proof as in
Theorem~\ref{maintheorem} shows that
$h((\mu(\Lambda_{1,n+1,t,b}))^{[p]}) -
h((\mu(\Lambda_{1,n,t,b}))^{[p]})$ approaches $h(\mu^{[p]}) =
ph(\mu)$ exponentially fast.
\end{proof}

\section{Interactions and Gibbs states}\label{gibbs}

In \cite{Pa}, similar techniques were used to show that the topological entropy of the $\mathbb{Z}^2$ hard square shift $\mathcal{H}$ is exponentially well approximable by differences of consecutive topological entropies of horizontal biinfinite strips. It turns out that this is a corollary of Theorem~\ref{maintheorem}; the unique measure of maximal entropy $\mu$ for $\mathcal{H}$ is in fact the unique MRF associated to a translation-invariant $\mathbb{Z}^2$-specification $\Lambda$ satisfying $q(\Lambda) < p_c$; in this case, for any finite set $S$ and $\delta \in \A^{\partial S}$,  $\Lambda^\delta$ will be uniform on configurations which are locally admissible in $S \cup \partial S$.
If one takes $t = b = 0^{\infty}$, then $\mu(\Lambda_{1,n,t,b})$ exists for all $n$, and is the unique measure of maximal entropy for the n.n. $\zz$-SFT composed of all locally admissible configurations on $H_{1,n}$ in $\mathcal{H}$.

We will use Theorem~\ref{maintheorem2} to generalize the main result of \cite{Pa} to some topological pressures by using some classical results of Ruelle regarding the relationship between equilibrium states and a class of $\zz^2$-MRFs called Gibbs states. 

Let $X$ be a nonempty n.n. $\zz^d$-SFT $X$ with language
$\L(X)$. We are mostly interested in the cases $d=1,2$.
An {\em interaction} on $X$ is simply a real-valued function
$\Phi$ on $\L(X)$. For finite $S$ and $x \in \L_S(X)$, define
$$
U_S(x) = U_S^\Phi(x) =\sum_{S' \subseteq S} \Phi(x|_{S'}).
$$
For finite $T \subseteq \zz^d $ such that $S \cap T = \varnothing$, $y
\in \L_T(X)$ and $xy \in \L_{S \cup T}(X)$, define:
$$
W_S(x,y) = W_S^\Phi(x,y) = \sum_{T' \subseteq S \cup T:~  T' \cap S
\ne \varnothing, ~T' \cap T \ne \varnothing} ~\Phi((xy)|_{T'}).
$$

A {\em Gibbs state} for an interaction $\Phi$ on $X$ is a measure $\mu$ with support contained in $X$ such that for any $x,y$ such that $xy \in \L_{S \cup T}(X)$ and $\mu(y) > 0$,
\begin{equation}\label{115}
\mu(x \ | \ y)  = \frac{e^{-U_S(x) -  W_S(x,y)}} {\sum_{\{x'
\in \L_S(X): ~x'y \in \L_{S \cup T}(X)\}}  e^{-U_S(x') -  W_S(x',y)}}
\end{equation}
This agrees with the classical definition of Gibbs state (see
Ruelle~\cite{ruelle}, Chapter 1, in particular equation (1.15)) in
the case of a n.n. SFT. Ruelle~\cite{ruelle} (Chapter 1)
shows that if $\Phi$ is bounded, then given any such $\Phi$ and $X$,
there is at least one Gibbs state $\mu$ (Ruelle's result is actually
much more general).

An interaction $\Phi$ is {\em translation-invariant} if it assigns
the same values to all translates of a given configuration on a
finite set. An interaction is a {\em nearest neighbor (n.n.) interaction}
if it vanishes on all configurations other than those on vertices
and edges.

We now
define a specification corresponding to (\ref{115}) for translation-invariant, n.n. interactions.  For a finite set
$S \subset \zz^d$ and $z \in \L_S(X)$, let
$$
\Phi(z,S) = \sum_{u \in S} \Phi(z(u)).
$$
For a finite set $E$ of edges in $E(\mathbb{Z}^d)$ such that
$\cup_{e \in E} ~e \subset S$ and  $z \in \L_S(X)$, let
$$
\Phi(z,E) = \sum_{e \in E} \Phi(z|_e).
$$

The {\em Gibbs $\mathbb{Z}^d$-specification determined by $\Phi$ and $X$}, denoted by $\Lambda_{\Phi,X}$,
is defined as follows. For any finite $S \subset \zz^d$, $x \in \A^S$, and $\delta \in \L_{\partial S}(X)$,
\begin{equation}
\label{Gibbs_spec}
\Lambda^\delta_{\Phi,X}(x) = \frac{\exp (- \Phi(x,S) - \Phi(x\delta,E(S,S) \cup E(S, \partial S)
) )} { \sum_{\{ x' \in \A^S: ~x'\delta \in \L_{S \cup \partial S}(X) \}}
\exp (- \Phi(x',S) - \Phi(x'\delta,E(S,S) \cup E(S, \partial S)) )}
\end{equation}
if $x\delta \in \L_{S \cup \partial S}(X)$, and $\Lambda^\delta_{\Phi,X}(x) = 0$ otherwise. For $\delta \in \A^{\partial S} \setminus
\L_{\partial S}(X)$, set $\Lambda^\delta_{\Phi,X} = \Lambda^{\delta'}_{\Phi,X}$ for some arbitrary
$\delta' \in \L_{\partial S}(X)$ (note that $\L_{\partial S}(X)$ is nonempty since $X$ is
nonempty). In this case $q(\Lambda_{\Phi,X})$ is determined only by
$\{\Lambda^\delta_{\Phi,X}\}_{\delta \in \L(X)}$.

If $T$ is finite, $\partial S \subseteq T \subset \zz^d \setminus S$,
$xy \in \L_{S \cup T}(X)$, and $\delta = y|_{\partial S}$, then
$$
U_S(x) = \Phi(x,S) + \Phi(x,E(S,S))
$$
$$
W_S(x,y) =  W_S(x,\delta) = \Phi(x\delta,E(S,\partial S)),
$$
and so for any Gibbs state $\mu$ for $\Phi$ on $X$, the conditional
probabilities (\ref{115}) reduce to $\mu(x \ | \ y) =
\Lambda_{\Phi,X}^\delta(x)$. Since $\mu$ is supported in $X$, all other
conditional probabilities $\mu(x \ | \ y)$ are spurious, and so $\mu$ is a $\mathbb{Z}^d$-MRF associated to $\Lambda_{\Phi,X}$.
Since there always exists a Gibbs state for $\Phi$ and $X$ (by
Ruelle's result above), $\Lambda_{\Phi,X}$ is valid. We also note that
since $\Phi$ was assumed to be translation-invariant, $\Lambda_{\Phi,X}$
is translation-invariant as well.
\medskip

Given a translation-invariant n.n. interaction $\Phi$
on a n.n. $\zz^2$-SFT $X$, we will find it useful to
represent a corresponding Gibbs state as a Gibbs state for an
interaction which is non-zero on a single finite set, namely the set
$\Delta = \{(0,0), (0,1), (1,0)\}$. Specifically, define
\begin{equation}
\label{hatphi} \hat{\Phi}(x) = \Phi(x|_{\{(0,0)\}}) +
\Phi(x|_{\{(0,0),(0,1)\}}) + \Phi(x|_{\{(0,0),(1,0)\}})
\end{equation}
for $x \in \L_\Delta(X)$.

\begin{proposition} \label{phiphihat} Any Gibbs state for $\hat \Phi$ on $X$ is a Gibbs state for $\Phi$ on $X$.
\end{proposition}

\begin{proof}
For a horizontal edge $e = \{u, u +(1,0)\}$ let $e^\perp =
\{u, u +(0,1)\}$.   For a vertical edge $e = \{u, u +(0,1)\}$ let
$e^\perp = \{u, u +(1,0)\}$.  For a finite set $S$, let $D(S) = \{\mbox{edges } e: e \cap S = \varnothing \mbox{ and } e^\perp \in
E(S,\partial S)\}$.

For finite $T$ such that $\partial S \subseteq T \subset \zz^2
\setminus S$, $x \in \A^S$, and
 $y \in \L_T(X)$, if $xy \in
\L_{S \cup T}(X)$, then letting $\delta = y|_{\partial S}$,   we have
$$
U^{\hat{\Phi}}_S(x) + W^{\hat{\Phi}}_S(x,y) = \Phi(x,S) + \Phi(x\delta,E(S,S) \cup E(S,\partial S)) +
\Phi(xy,D(S))
$$
Note that the last term in this expression is the same if we replace
$x$ by any $x' \in \A^S$ such that $x'\delta \in \L_{S \cup
\partial S}(X)$.   It follows that for a Gibbs state for $\hat{\Phi}$ on
$X$, the conditional probabilities (\ref{115}) reduce to the
specification determined by $\Phi$ and $X$.

\end{proof}

We conclude this section by presenting two interactions on n.n. $\zz^2$-SFTs which define historically important Gibbs measures. We will return to these examples later in the paper to demonstrate how later results apply to them.

\medskip

1. The $\zz^2$ \textit{hard-core model} with activity $a$ is given by $X$ equal to the $\zz^2$ hard square shift $\mathcal{H}$ and translation-invariant n.n. interaction $\Phi$ defined by $\Phi(u) = au(v)$ for $u$ a configuration on a vertex $v$ and $\Phi(u) = 0$ for $u$ a configuration on an edge. In this model, $a$ can be thought of as the ``weight'' given to the symbol $1$.

2. The $\zz^2$ \textit{Ising antiferromagnet} with external magnetic field $h$ and temperature $T = \frac{1}{\beta}$ is defined by $X = \{\pm 1\}^{\zz^2}$ and translation-invariant n.n. interaction $\Phi$ defined by $\Phi(u) = -\beta hu(v)$ for $u$ a configuration on a vertex $v$ and $\Phi(u) = \beta u(v_1) u(v_2)$ for $u$ a configuration on an edge $\{v_1, v_2\}$. In this model, $h$ is an external influence which gives individual sites a preference between $1$ and $-1$, and $\beta$ can be thought of as the ``penalty'' imposed to aligned adjacent sites; when $\beta$ is very large, adjacent sites are more likely to differ.

\section{Strongly Irreducible SFTs}\label{SI}

We from now on restrict our attention to a specific class of n.n. $\zz^2$-SFTs
with very strong topological mixing properties.

\begin{definition}
A n.n. $\zz^2$-SFT $X$ is {\rm strongly irreducible} (with filling distance $L$) if
for any finite $S,T \subset \zz^2$ with $d(S,T) > L$, and any $x \in \L_S(X), y \in \L_T(X)$,
it is always the case that $xy \in \L_{S\cup T}(X)$.
\end{definition}

Our first use of strong irreducibility is to present a sufficient condition for a $\mathbb{Z}^2$-MRF
to be fully supported within an SFT.

\begin{proposition}\label{full_support}
If $X$ is a strongly irreducible n.n. $\mathbb{Z}^2$-SFT
and $\Lambda$ is a valid $\mathbb{Z}^2$-specification such that
$x\delta \in \L(X) \Longrightarrow \Lambda^{\delta}(x) > 0$, then any
$\mathbb{Z}^2$-MRF associated to $\Lambda$ whose support is contained in $X$
is fully supported on $X$.
\end{proposition}

\begin{proof}
Denote by $L$ the filling distance of $X$. Let $\mu$ be any $\zz^2$-MRF associated to $\Lambda$
whose support is contained in $X$. Fix any $w \in
\L_{[-n,n]^2}(X)$ and any $\delta \in
\A^{\partial([-n-L,n+L]^2)}$ such that $\delta$ has positive $\mu$-measure. Since
the support of $\mu$ is contained in $X$, $\delta \in \L(X)$.
Therefore, by strong irreducibility of $X$, there exists $x \in
\L_{[-n-L,n+L]^2}(X)$ with $x|_{[-n,n]^2} = w$ and $x\delta \in \L(X)$.
By the assumption on $\Lambda$,
$\Lambda^{\delta}(x) > 0$, and since $\mu$ is associated to
$\Lambda$, this means that $\mu^{\delta}(x) > 0$. But
then since $\mu(\delta)$ is positive, $\mu(x)$ is as well.
Since $w$ is a subconfiguration of $x$, $\mu(w) > 0$.

\end{proof}

From this, we can conclude a fact about certain valid $\mathbb{Z}^2$-specifications supported on strongly irreducible n.n. $\mathbb{Z}^2$-SFTs that will be useful later.

\begin{proposition}\label{HZ}
Let $X$ be a strongly irreducible n.n. $\mathbb{Z}^2$-SFT
and $\Lambda$ be a valid $\mathbb{Z}^2$-specification such that
$x\delta \in \L(X) \Longrightarrow \Lambda^{\delta}(x) > 0$.
Assume that there exists a $\mathbb{Z}^2$-MRF associated to $\Lambda$ whose support
is contained in $X$ and that
$q(\Lambda) < 1$. Then for any rectangle $R$ and $\delta, \delta' \in \L_{\partial R}(X)$, there exists $u \in \A^R$ such that $\delta u, \delta' u \in \L(X)$.
\end{proposition}

\begin{proof}
Consider such $X$, $\Lambda$, $R$, $\delta$, and $\delta'$, and let $\mu$ be a $\mathbb{Z}^2$-MRF associated to $\Lambda$
with support contained in $X$.
By Proposition~\ref{full_support}, $\mu$ is fully supported on $X$.

Without loss of generality, we assume $R = [0,m] \times [0,n]$ for some nonnegative $m,n$. We begin by proving the proposition for $m = n = 0$, i.e. $R = \{(0,0)\}$.  Consider any $\zeta, \zeta' \in \L_{N_{(0,0)}}(X)$. Then by full support of $\mu$ on $X$, $\mu(\zeta), \mu(\zeta') > 0$. Since $q(\Lambda) < 1$, there exists $a \in \A^{\{(0,0)\}}$ such that $\Lambda^{\zeta}(a), \Lambda^{\zeta'}(a) > 0$, and since $\mu$ is associated to $\Lambda$, $\mu^{\zeta}(a), \mu^{\zeta'}(a) > 0$. Therefore, $\mu(\zeta a), \mu(\zeta' a) > 0$, and since the support of $\mu$ is contained in $X$, $\zeta a, \zeta' a \in \L(X)$.

We will use this fact to deal with general $R$. Define by $r_1, r_2, \ldots, r_{mn}$ the elements of $R$, listed in lexicographic order (i.e. the bottom row from left to right, then the next row from left to right, etc.) For each $k \in [0,mn]$, define $R_k = \{r_i\}_{i=1}^k$. We define $u$ letter by letter, on the sites $r_i$ in order. Suppose that for some $k \in [0,mn)$, we have defined $u_k \in \A^{R_k}$ such that $\delta u_k, \delta' u_k \in \L(X)$. We must define $u_{k+1} \in \A^{R_{k+1}}$ such that $u_{k+1}|_{R_k} = u_k$ and $\delta u_{k+1}, \delta' u_{k+1} \in \L(X)$.

Define $C = N_{r_{k+1}} \setminus (R_k \cup \partial R)$, the set of neighbors of $r_{k+1}$ which do not have letters assigned to them in $\delta$, $\delta'$, or $u_k$. Clearly, since $\delta u_k$ and $\delta' u_k$ are globally admissible in $X$, there exist $\eta, \eta' \in \L_C(X)$ such that $\delta \eta u_k, \delta' \eta' u_k \in \L(X)$. Define $\zeta = (\delta \eta u_k)|_{N_{r_{k+1}}}$ and $\zeta' = (\delta' \eta' u_k)|_{N_{r_{k+1}}}$; clearly $\zeta, \zeta' \in \L(X)$.

Then by the proof of the proposition for $R$ consisting of a single site, there exists $a \in \L_{\{r_{k+1}\}}(X)$ such that $\zeta a, \zeta' a \in \L(X)$. Then from the facts that $X$ is a n.n. SFT, that $\zeta, \zeta' \in \A^{N_{r_{k+1}}}$, and that $\delta \eta u_k$ and $\delta' \eta' u_k$ are globally admissible in $X$, we can conclude that $\delta \eta u_k a$ and $\delta' \eta' u_k a$ are globally admissible in $X$ as well. But then trivially $\delta u_k a, \delta' u_k a \in \L(X)$, and so we can take $u_{k+1} = u_k a$.

This inductive process eventually yields $u_{mn} \in \A^R$ for which $\delta u_{mn}, \delta' u_{mn} \in \L(X)$, and so by taking $u = u_{mn}$ we are done.

\end{proof}

Our main application of strong irreducibility is, for any
Gibbs specification $\Lambda_{\Phi,X}$, to
guarantee the existence of periodic rows $t,b$ such that
$(\Lambda_{\Phi,X})_{m,n,t,b}$ is valid for $n-m$ large.


We first recall the result of Ward~\cite{ward} that any strongly
irreducible n.n. $\zz^2$-SFT $X$ has a globally admissible periodic
row (i.e., a periodic configuration on $\zz \times \{0\}$ which
extends to an element of $X$); in fact, every such SFT has a doubly
periodic element, though we will not need this fact here.

For integers $m < n$ and globally admissible periodic rows $t,b$,
let $X_{m,n,t,b}$ be the set of all configurations $x \in \A^{\zz
\times [m,n]}$ such that $txb$ is locally admissible. Note that all
elements of $X_{m,n,t,b}$ are in fact globally admissible.

\begin{proposition}\label{periodic}
Let $X$ be a strongly irreducible n.n. $\zz^2$-SFT and
$\Phi$ be a translation-invariant n.n. interaction on $X$.
Let $t$ and $b$ be periodic rows which are globally admissible in $X$ (which always
exist by~\cite{ward}). If $n-m$ exceeds the filling distance of $X$,
then $m,n,t,b$ is compatible with $\Lambda_{\Phi,X}$ (and therefore
by Proposition~\ref{stripvalidity} $(\Lambda_{\Phi,X})_{m,n,t,b}$ is
valid).

Moreover, there is an MRF associated to
$(\Lambda_{\Phi,X})_{m,n,t,b}$ supported in  $X_{m,n,t,b}$.

\end{proposition}

\begin{proof}
Let $t$ and $b$ be such rows. Since $\Lambda_{\Phi,X}$ is Gibbs, it is valid, and so
has an associated MRF $\mu$
supported in $X$.
Let $m,n$ be chosen so that $n-m$
exceeds the filling distance of $X$. Then for all $k$, there exists $w_k \in \L_{[-k,k] \times [m-1,n+1]}(X)$
whose top row is $t|_{[-k,k]}$ and bottom row is $b|_{[-k,k]}$.
Proposition~\ref{full_support} applies to $\mu$, and so $\mu(w_k) >
0$.
Thus, $m,n,t,b$ is compatible with $\Lambda_{\Phi,X}$ (by taking
$\delta_k = w_k|_{\partial ([-k+1,k-1] \times [m,n])}$ in the
definition of compatibility).

The measure obtained as in Proposition~\ref{stripvalidity} from the
$\delta_k$ is clearly supported in $X_{m,n,t,b}$.

\end{proof}

For any translation-invariant n.n. interaction $\Phi$, strongly irreducible n.n.
$\zz^2$-SFT $X$ with filling distance $L$, globally admissible
periodic $t$ and $b$, and $m,n$ with $n-m > L$, for notational
convenience we denote by $\Lambda_{\Phi,X,m,n,t,b}$ the
specification $(\Lambda_{\Phi,X})_{m,n,t,b}$ which by
Proposition~\ref{periodic} is valid. If in addition $q(\Lambda) < 1$, then by
Proposition~\ref{stripuniqueness} there is a unique $H_{m,n}$-MRF
$\mu(\Lambda_{\Phi,X,m,n,t,b})$ associated to
$\Lambda_{\Phi,X,m,n,t,b}$, which we call the {\em induced Gibbs
state} for $\Phi$ and $X$, and which, again for convenience, we
denote by $\mu_{\Phi,X,m,n,t,b}$.
By Proposition~\ref{periodic} $\mu_{\Phi,X,m,n,t,b}$ is supported in
$X_{m,n,t,b}$.

\section{Induced Gibbs states as one-dimensional Markov Chains}\label{markov}

Our eventual goal is to use the results of Sections~\ref{MRFuniqueness},~\ref{gibbs}, and~\ref{SI} to give an algorithm for approximating certain topological pressures by means of induced Gibbs states on strips. For this procedure to be useful though, it must be the case that these induced Gibbs states are tractable measures to deal with. Under some elementary assumptions, this does turn out to be the case; in fact the induced Gibbs states turn out to be Markov chains.

Let $X$ be a strongly irreducible n.n. $\zz^2$-SFT with filling distance $L$, $\Phi$ a translation-invariant n.n.
interaction on $X$ such that $q(\Lambda_{\Phi,X}) < 1$, and let $\mu$ be any Gibbs state for $\Phi$ on $X$. Let $t$ and $b$ be globally admissible periodic rows for $X$, which for now we assume to be constant (though such rows may not exist in general): $t = t_0^\infty, b=b_0^\infty$. Let $m,n$ be integers such that $n-m>L$.

%
%

Let $\C_{m,n,t,b}$ denote the set of all ``locally
admissible columns compatible with $t$ and $b$'' on $H_{m,n}$, i.e., all $x_m \ldots x_n$  such
that $x_i x_{i+1} \in \E_2$ for $i=m, \ldots, n-1$, and $b_0 x_m, x_n t_0
\in \E_2$.
Let $\E_{m,n,t,b}$ denote the set of all locally
admissible $(n-m+1) \times 2$ rectangles, i.e., ordered pairs of
columns $(x_m \ldots x_{n}, y_m \ldots y_{n}) \in (\C_{m,n,t,b})^2$ such
that $x_i y_i \in \E_1$ for each $i=m, \ldots, n$.

Observe that $X_{m,n,t,b}$ (defined near the end of
Section~{\ref{SI}) is the n.n. $\zz$-SFT on the alphabet
$\C_{m,n,t,b}$ defined by $\E = \E_{m,n,t,b}$. Here, we are
identifying a configuration on a finite
interval
of $\zz$ over the alphabet $\C_{m,n,t,b}$ with the
corresponding configuration on a finite rectangle in $\zz \times
[m,n]$ over the alphabet $\A$. Recall that since $t$ and $b$ are
globally admissible and $X$ is a n.n. $\zz^2$-SFT, we have
$\L(X_{m,n,t,b}) \subset \L(X)$.

Denote by $\overline{\Lambda}_{\Phi,X,m,n,t,b}$ the restriction of the specification $\Lambda_{\Phi,X,m,n,t,b}$ to configurations $\delta$ of the form $\partial(H_{m,n},F \times [m,n])$ for finite sets $F$. We think of $\overline{\Lambda}_{\Phi,X,m,n,t,b}$ as a $\mathbb{Z}$-specification over the alphabet $\A^{[m,n]}$.

\begin{corollary}\label{HZ2}
If $X$ is a strongly irreducible n.n. $\zz^2$-SFT, $\Phi$ is a translation-invariant n.n. interaction for which $q(\Lambda_{\Phi,X}) < 1$, $t,b$ are globally admissible constant sequences, and $n-m$ exceeds the filling distance of $X$, then $q(\overline{\Lambda}_{\Phi,X,m,n,t,b}) < 1$.
\end{corollary}

\begin{proof}
To show that $q(\overline{\Lambda}_{\Phi,X,m,n,t,b}) < 1$, it suffices to show that for any $c,c' \in \L_{\{-1\}}(X_{m,n,t,b})$ and $d,d' \in \L_{\{1\}}(X_{m,n,t,b})$ such that $cd, c'd' \in \L(X_{m,n,t,b})$, there exists $e \in \L_{\{0\}}(X_{m,n,t,b})$ such that $ced, c'ed' \in \L(X_{m,n,t,b})$. (This is sufficient since $\Lambda_{\Phi,X}$, being a Gibbs specification, has the property that $\Lambda_{\Phi,X}^{\delta}(u) > 0$ if $u \delta \in \L(X)$.)

However, note that this is equivalent to showing that for any $\delta, \delta' \in \L_{\partial (\{0\} \times [m,n])}(X)$ with $t_0$ at the top and $b_0$ at the bottom, there exists $u \in \A^{\{0\} \times [m,n]}(X)$ such that $\delta u, \delta' u \in \L(X)$. This is a straightforward consequence of Proposition~\ref{HZ}, which can be applied since $\Lambda_{\Phi,X}$ is the Gibbs specification determined by $\Phi$ and $X$ and $q(\Lambda_{\Phi,X}) < 1$.

\end{proof}

\medskip

\begin{corollary}\label{mixing}
If $X$ is a strongly irreducible n.n. $\zz^2$-SFT, $\Phi$ is a translation-invariant n.n. interaction for which $q(\Lambda_{\Phi,X}) < 1$, $t,b$ are globally admissible constant sequences, and $n-m$ exceeds the filling distance of $X$, then $X_{m,n,t,b}$ is a mixing $\zz$-SFT.
\end{corollary}

\begin{proof}
By Corollary~\ref{HZ2}, $q(\overline{\Lambda}_{\Phi,X,m,n,t,b}) < 1$. We prove that $X_{m,n,t,b}$ is mixing by using some well-known facts about the structure of n.n. $\zz$-SFTs. In particular, we show that $X_{m,n,t,b}$ is irreducible and aperiodic; for more details on these properties see \cite{Lind-Marcus}.

We first show that $X_{m,n,t,b}$ is irreducible. Assume for a contradiction that it is not. Then there exist at least two nontrivial irreducible components $C, D \subseteq \mathcal{C}_{m,n,t,b}$. But then clearly we have a contradiction to $q(\overline{\Lambda}_{\Phi,X,m,n,t,b}) < 1$; a boundary configuration consisting of two letters from $C$ can only be filled in a globally admissible way with a letter from $C$, and the same is true for $D$. Since $C$ and $D$ were nontrivial components, there exist such boundary configurations $\delta, \delta'$ which are globally admissible, and since $C \cap D = \varnothing$ and $\Lambda_{\Phi,X}$ is supported on $X$, $d(\overline{\Lambda}_{\Phi,X,m,n,t,b}^{\delta}, \overline{\Lambda}_{\Phi,X,m,n,t,b}^{\delta'}) = 1$. Therefore, our original assumption was wrong and $X_{m,n,t,b}$ is irreducible.

It remains to show that $X_{m,n,t,b}$ is aperiodic, which is done in the same way; suppose for a contradiction that $X_{m,n,t,b}$ can be partitioned into period classes $P_1, \ldots, P_k$, $k>1$. Then if we take $\delta$ to be a globally admissible boundary configuration in $X_{m,n,t,b}$ consisting of a letter from $P_{k}$ on the left and a letter from $P_{2}$ on the right, and $\delta'$ to be a globally admissible boundary configuration in $X_{m,n,t,b}$ consisting of a letter from $P_1$ on the left and a letter from $P_{3 \pmod k}$ on the right, then $\delta$ and $\delta'$ can only be filled with letters from $P_1$ and $P_2$ respectively. Again, since $P_1 \cap P_2 = \varnothing$, this contradicts $q(\overline{\Lambda}_{\Phi,X,m,n,t,b}) < 1$, so $X_{m,n,t,b}$ is aperiodic and irreducible, therefore mixing.

\end{proof}

The induced Gibbs state $\mu_{\Phi,X,m,n,t,b}$ can be viewed as a
measure on $(\C_{m,n,t,b})^{\zz}$ supported in $X_{m,n,t,b}$, and
when viewed in this way, it is a $\mathbb{Z}$-MRF associated to
$\overline{\Lambda}_{\Phi,X,m,n,t,b}$. We will show that when viewed
in this way $\mu_{\Phi,X,m,n,t,b}$ is a translation-invariant
irreducible 1st-order Markov chain, with a transition probability
matrix defined explicitly in terms of $\Phi$.

We first describe $\overline{\Lambda}_{\Phi,X,m,n,t,b}$ explicitly in terms of $\Phi$.
%
For a configuration $z$ on a finite set
$U \subset \zz \times \{n\}$ let
$$
\Phi_+(z,U) = \sum_{u \in U} \Phi\left(\substack{t_0\\[4pt] z_u}\right)
$$
For a configuration $z$ on a finite set $U \subset \zz \times \{m\}$ let
$$
\Phi_-(z,U) = \sum_{u \in U} \Phi\left(\substack{z_u\\[4pt] b_0}\right)
$$
Let $R=[-k,k] \times [m,n]$, $x \in \A^R$, $\delta \in \A^{
\partial(R,H_{m,n}) }$ such that $x\delta \in \L(X_{m,n,t,b})$. From
(\ref{Gibbs_spec}) and (\ref{mntbdefn}), we have:
\begin{equation}
\label{Gibbs_spec_strip} \overline{\Lambda}_{\Phi,X,m,n,t,b}^{\delta}(x) =
\frac{\exp(-A(x,\delta))}{ \sum_{\{w \in \A^R: ~w\delta \in
\L(X_{m,n,t,b})  \}}\exp(-A(w,\delta))}
\end{equation}
where
\begin{multline}\label{defnAdot}
A(z,\delta) =
\Phi(z,R) + \Phi(z\delta,E(R,R) \cup
E(R, \partial(R,H_{m,n}))  )\\  +  \Phi_+( z,R \cap (\zz \times \{n\})) + \Phi_-(z,R
\cap (\zz \times \{m\})).
\end{multline}

We claim that $\overline{\Lambda}_{\Phi,X,m,n,t,b}$ can be expressed as a
$\mathbb{Z}$-specification determined by a n.n. interaction $\overline{\Phi}_{m,n,t,b}$ on $X_{m,n,t,b}$.
We define $\overline{\Phi}_{m,n,t,b}$ to vanish on all finite configurations other than
those on edges in $\E_{m,n,t,b}$, and
on such edges, it is defined by
\begin{multline}\label{fmn}
\overline{\Phi}_{m,n,t,b}(x_m \ldots x_{n}, y_m \ldots y_{n} ) =\\
\left(\left(\sum_{i=m}^n \Phi(x_i)\right) +
\Phi\left(\substack{x_m\\[4pt] b_0}\right) + \left(\sum_{i=m}^{n-1}
\Phi\left(\substack{x_{i+1}\\[4pt] x_i}\right)\right) +
\Phi\left(\substack{t_0\\[4pt] x_n}\right) + \left(\sum_{i=m}^{n}
\Phi(x_i,y_i)\right) \right).
\end{multline}
Let $R=[-k,k] \times [m,n]$, $x \in \A^R$, $\delta \in \A^{
\partial(R,H_{m,n}) }$ such that $x\delta \in \L(X_{m,n,t,b})$.
Observe that
\begin{equation}
\label{1d}
\Lambda_{\overline{\Phi}_{m,n,t,b},X_{m,n,t,b}}^\delta(x) = \frac{\exp (- B(x,\delta))} {  \sum_{\{w \in \A^R: ~w\delta \in
\L(X_{m,n,t,b})  \}}  \exp (- B(w,\delta))                }
\end{equation}
where
\begin{multline}
\label{Bz}
B(z,\delta) = \\ \sum_{j=-k}^{k-1} \left(\overline{\Phi}_{m,n,t,b}(z|_{\{j\}\times[m,n]},z|_{\{j+1\}\times[m,n]}) +
\overline{\Phi}_{m,n,t,b}(z|_{\{k\}\times[m,n]},\delta|_{\{k+1\}\times[m,n]})\right)\\
+ \overline{\Phi}_{m,n,t,b}(\delta|_{\{-k-1\} \times [m,n]}, z|_{\{-k\} \times [m,n]})
= A(z,\delta) + C(\delta),
\end{multline}
where
\begin{multline*}
C(\delta) =
\left(\sum_{i=m}^n \Phi(\delta(-k-1,i))\right) +
\Phi\left(\substack{\delta(-k-1,m)\\[4pt] b_0}\right) \\
+ \left(\sum_{i=m}^{n-1} \Phi\left(\substack{ \delta(-k-1,i+1)  \\[4pt]   \delta(-k-1,i)  }  \right)\right) +
\Phi\left(\substack{t_0\\[4pt] \delta(-k-1,n)  }\right).
\end{multline*}
%
Comparing (\ref{Gibbs_spec_strip}), (\ref{defnAdot}), (\ref{1d}) and (\ref{Bz}), we see that
$$
\Lambda_{\overline{\Phi}_{m,n,t,b},X_{m,n,t,b}}^\delta = \overline{\Lambda}_{\Phi,X,m,n,t,b}^{\delta}
$$
for all $\delta \in \L_{\partial(R,H_{m,n})}(X_{m,n,t,b})$.
Since $\mu_{\Phi,X,m,n,t,b}$ is supported in $X_{m,n,t,b}$, it
follows that $\mu_{\Phi,X,m,n,t,b}$, when viewed as a
$\mathbb{Z}$-MRF, is a Gibbs state for the interaction
$\overline{\Phi}_{m,n,t,b}$ on $X_{m,n,t,b}$.


Let $A_{\Phi,X,m,n,t,b}$ be the square matrix indexed by $\C_{m,n,t,b}$ and defined by
\begin{equation}
\label{defnA}
(A_{\Phi,X,m,n,t,b})_{c,d} =
\begin{cases}
e^{-\overline{\Phi}_{m,n,t,b}(c,d)} & \textrm{if } (c,d) \in \E_{m,n,t,b}\\
0 & \textrm{ otherwise.}
\end{cases}
\end{equation}
We will frequently suppress the dependence of $A$ on $\Phi,X,m,n,t,b$ when it causes no confusion.

By Corollary~\ref{mixing}, we may assume (by deleting elements of $\C_{m,n,t,b}$ which do not actually
appear in $X_{m,n,t,b}$) that $A$ is a primitive 
matrix. Let $\lambda(A)$ represent the Perron (i.e., largest) eigenvalue of
$A$. By Perron-Frobenius Theory, there are unique (up to scalar multiples)
right and left (positive) eigenvectors $u = u_{\Phi,X,m,n,t,b}$ and $v = v_{\Phi,X,m,n,t,b}$
corresponding to $\lambda(A)$.

%

Let $\Pi = \Pi_{\Phi,X,m,n,t,b}$ (again, dependence will frequently be suppressed) be the
(primitive) probability transition matrix indexed by
$\C_{m,n,t,b}$ defined by
\begin{equation}
\label{Pi}
\Pi_{c,d} = \frac{A_{c,d} v(d)}{\lambda(A)v(c)}
\end{equation}

\begin{proposition}\label{Markov}
Let $\Phi$ be a translation-invariant n.n. interaction on a strongly irreducible n.n. $\zz^2$-SFT
$X$.  Let $t$ and $b$ be globally admissible constant rows and assume that
$n-m$ exceeds the filling distance of $X$.
Assume that $q(\Lambda_{\Phi,X}) <1$.
Then the induced Gibbs state $\mu_{\Phi,X,m,n,t,b}$ for $\Phi$ and $X$
is a one-dimensional translation-invariant mixing 1st-order Markov chain
with probability transition matrix $\Pi_{\Phi,X,m,n,t,b}$. 
\end{proposition}

\begin{proof}  This is a special case of a much more general result~\cite{Georgii} (Theorem 10.25).
However, the proof in our case is much simpler, as follows.

Write $\mu = \mu_{\Phi,X,m,n,t,b}$,   $u = u_{\Phi,X,m,n,t,b}$, $v = v_{\Phi,X,m,n,t,b}$
and  $\lambda = \lambda(A_{\Phi,X,m,n,t,b})$.  We assume that $u$ and $v$ are normalized
so that $u \cdot v =1$. Fix any $k>0$.

Since $\mu$ is a Gibbs state for the interaction
$\overline{\Phi}_{m,n,t,b}$ on  $X_{m,n,t,b}$,
for any positive integers $k,\ell$ and any $x_{-k}, \ldots, x_{-1}, x_0 \in \C_{m,n,t,b}$,
$$
\mu(x_0|x_{-1}, \ldots, x_{-k}) =
\sum_{x_\ell} \mu(x_\ell |x_{-1}, \ldots, x_{-k})
\mu(x_0|x_\ell, x_{-1}, \ldots, x_{-k})
$$
$$
=
\sum_{x_\ell} \mu(x_\ell |x_{-1}, \ldots, x_{-k}) \sum_{x_1, \ldots, x_{\ell-1}}  \mu(x_0, x_1, \ldots, x_{\ell-1}|x_\ell, x_{-1}, \ldots, x_{-k})
$$
$$
=
\sum_{x_\ell} \mu(x_\ell |x_{-1}, \ldots, x_{-k}) \sum_{x_1, \ldots, x_{\ell-1}}  \mu(x_0, x_1, \ldots, x_{\ell-1}|x_\ell, x_{-1})
$$
$$
=
\sum_{x_\ell} \mu(x_\ell |x_{-1}, \ldots, x_{-k}) \sum_{x_1, \ldots, x_{\ell-1}}
\frac{A_{x_{-1}x_0} A_{x_0x_1} \cdots A_{x_{\ell-1}x_\ell}}
{ \sum_{ x_0', x_1', \ldots, x_{\ell-1}'} A_{x_{-1}x_0'} A_{x_0'x_1'} \cdots A_{x_{\ell-1}'x_\ell} }
$$
$$
=
\sum_{x_\ell} \mu(x_\ell |x_{-1}, \ldots, x_{-k}) \frac{A_{x_{-1}x_0} (A^\ell)_{x_0x_\ell}   }{ (A^{\ell+1})_{x_{-1}x_\ell}}
$$

Since $A$ is primitive, by~\cite{Lind-Marcus} (Theorem 4.5.12) we have that
$\lim_{\ell \rightarrow \infty} \frac{(A^\ell)_{c,d}}{\lambda^\ell} = v_c u_d$.
Thus, given $\epsilon >0$, for sufficiently large $\ell$,
$\mu(x_0|x_{-1}, \ldots, x_{-k})$ is within $\epsilon$ of
$$
\sum_{x_\ell} \mu(x_\ell |x_{-1}, \ldots, x_{-k})
\frac{A_{x_{-1}x_0} v_{x_0}u_{x_\ell}} {v_{x_{-1}}u_{x_\ell} \lambda}
=
\frac{A_{x_{-1}x_0} v_{x_0}} {v_{x_{-1}} \lambda}
$$
Thus, $\mu(x_0|x_{-1}, \ldots, x_{-k}) =  \frac{A_{x_{-1}x_0} v_{x_0}} {v_{x_{-1}}\lambda} = \Pi_{x_{-1}, x_0}$.
In particular, $\mu$ is a translation-invariant mixing 1st-order Markov chain.

\end{proof}

\section{Pressure and equilibrium states}\label{pressure}


We now turn to our main application of Theorem~\ref{maintheorem2}: the approximation of
certain topological pressures on strongly irreducible n.n. $\zz^2$-SFTs.

We recall that for any n.n. $\zz^d$-SFT $X$ and $f \in C(X)$,
an equilibrium state is a translation-invariant measure $\mu$ on $X$ for which
$h(\mu) + \int f \ d\mu$ is maximized, and that this maximum $P(f)$ is called the topological pressure
of $f$ on $X$. Our main tool is Theorem 4.2 from \cite{ruelle}, which
proves that any equilibrium state is a Gibbs state.

Let $X$ be a strongly irreducible n.n. $\zz^2$-SFT $X$
and $\Phi$ be a translation-invariant n.n. interaction
such that $q(\Lambda_{\Phi,X}) < p_c$. Define $\hat{\Phi}$ as in (\ref{hatphi}). Let
$f_\Phi: X \rightarrow R$ be defined by $f_{\Phi}(x) =
-\hat{\Phi}(x|_{\Delta })$. In this case, Theorem 4.2 from \cite{ruelle}
shows that any equilibrium state for $f_{\Phi}$ is a Gibbs state for
$\hat{\Phi}$. (In fact, this is the reason for defining $\hat{\Phi}$:
technically Theorem 4.2 from \cite{ruelle} applies only to interactions supported on
configurations on a single shape.) According to our Proposition~\ref{phiphihat}, any
equilibrium state for $f_{\Phi}$ is a Gibbs state for $\Phi$ as well.

Since $X$ is strongly irreducible, there exist $t,b$ globally admissible periodic
rows in $X$. We for now assume that $t,b$ are constant, and deal with the general periodic case later.
For any $n$ which exceeds the filling distance of $X$,
let $\lambda_n = \lambda(A_{\Phi,X,1,n,t,b})$.

\begin{theorem}
\label{main_Gibbs} Let $X$ be a strongly irreducible n.n. $\zz^2$-SFT and
$\Phi$ be a translation-invariant n.n. interaction on $X$ such that
$q(\Lambda_{\Phi,X}) < p_c$. Let $t$ and $b$ be globally admissible constant rows.
Then 
there exist constants $Q, R >0$ such that for 
sufficiently large $n$,
$$
|\log \lambda_{n+1} - \log \lambda_{n} - P(f_{\Phi})| < Qe^{-Rn}.
$$
\end{theorem}

\begin{proof}
Let $\mu$ be an equilibrium state for $f_{\Phi}$, i.e.,  $P(f_{\Phi}) = h(\mu) +
\int f_{\Phi} d\mu$. By the discussion above, $\mu$ is a Gibbs state for $\Phi$ on $X$ and,
by Theorem~\ref{uniqueness}, is the unique $\mathbb{Z}^2$-MRF associated to the $\mathbb{Z}^2$-specification $\Lambda = \Lambda_{\Phi,X}$.

Combining Proposition~\ref{Markov} with the well-known characterization of unique equilibrium states of locally constant functions
as Markov chains~(see \cite[p. 99]{krieger}, \cite{blascamp}), we see that for any $n$ larger than the filling distance of $X$,
$\mu_{\Phi,X,1,n,t,b}$ is the unique equilibrium state
for $\overline{\Phi}_{1,n,t,b}$ on the $\zz$-SFT $X_{1,n,t,b}$, and $\log \lambda_n = P(\overline{\Phi}_{1,n,t,b}) = h(\mu_{\Phi,X,1,n,t,b}) + \int
\overline{\Phi}_{1,n,t,b} d\mu_{\Phi,X,1,n,t,b}$.

By Theorem~\ref{maintheorem} and Proposition~\ref{periodic}, there exist constants $\overline{Q},\overline{R}$ such that for sufficiently large $n$,
\[
\left| h(\mu_{\Phi,X,1,n+1,t,b}) - h(\mu_{\Phi,X,1,n,t,b}) - h(\mu)\right| < \overline{Q}e^{-\overline{R}n}.
\]
It remains to show that $\int \overline{\Phi}_{1,n+1,t,b} \ d\mu_{\Phi,X,1,n+1,t,b} - \int \overline{\Phi}_{1,n,t,b} \ d\mu_{\Phi,X,1,n,t,b}$
converges exponentially fast to $\int f_{\Phi} \ d\mu$. Recalling the definition of $\overline{\Phi}_{1,n,t,b}$ ((\ref{fmn})), we see that
$\int \overline{\Phi}_{1,n+1,t,b} d\mu_{\Phi,X,1,n+1,t,b} - \int \overline{\Phi}_{1,n,t,b} d\mu_{\Phi,X,1,n,t,b}$
can be decomposed into a sum of the form
\begin{multline}\label{phidiffdecomp}
\sum \left( \int F(x) d\mu_{\Phi,X,1,n+1,t,b} - \int F(x) d\mu_{\Phi,X,1,n,t,b}\right)\\
+ \sum \left( \int F(\sigma_{(0,1)}x')) d\mu_{\Phi,X,1,n+1,t,b} - \int F(x') d\mu_{\Phi,X,1,n,t,b}\right)\\
+ \int (f_{\Phi} \circ \sigma_{(0,-\lfloor n/2 \rfloor)}) d\mu_{\Phi,X,1,n+1,t,b},
\end{multline}
where each $x$ in the first sum is a configuration with shape a vertex or edge contained in $\{0,1\} \times
[1,\lfloor n/2 \rfloor]$, each $x'$ in the second sum is a configuration with shape a vertex or edge
contained in $\{0,1\} \times [\lfloor n/2 \rfloor, n]$, and $F(x)$ can represent any of the functions
$\Phi(x)$, $\Phi\left(\substack{x\\[4pt]b_0}\right)$, or $\Phi\left(\substack{t_0\\[4pt]x}\right)$.
(Clearly, in the latter two cases, $x$ must be a configuration on a single site contained in the
bottom or top row respectively.)

Since distribution distance is dominated by $\dbar$ distance,
by Theorem~\ref{dbarclose} there exist $K, L > 0$ such that for any
configuration $x$ with shape a vertex 
or edge contained in $\{0,1\} \times
[1,\lfloor n/2 \rfloor]$,
\begin{equation}\label{nonshiftedweakbound}
|\mu_{\Phi,X,1,n+1,t,b}(x) - \mu_{\Phi,X,1,n,t,b}(x)| \le Ke^{-Ln/2}.
\end{equation}

Similarly, for any configuration $x$ with shape a vertex or edge contained in $\{0,1\} \times [\lfloor n/2 \rfloor, n]$,
\begin{equation}\label{shiftedweakbound}
|\mu_{\Phi,X,1,n+1,t,b}(\sigma_{(0,1)}x) - \mu_{\Phi,X,1,n,t,b}(x)| \le Ke^{-Ln/2}.
\end{equation}

By (\ref{nonshiftedweakbound}) and (\ref{shiftedweakbound}), if we take $Q' = K \max_x |\Phi(x)|$, then each of the first two sums in (\ref{phidiffdecomp}) is less than $5n Q' e^{-Ln/2}$.

The third term of (\ref{phidiffdecomp}) converges to $\int f_{\Phi} d\mu$ by Proposition~\ref{weaklimit}, and is exponentially Cauchy by
(\ref{nonshiftedweakbound}) and (\ref{shiftedweakbound}). Therefore, there exists $Q''$ such that for sufficiently large $n$,
\[
\left| \int (f_{\Phi} \circ \sigma_{(0,-\lfloor n/2 \rfloor)}) d\mu_{\Phi,X,1,n+1,t,b} - \int f_{\Phi} d\mu \right| \leq Q'' e^{-Ln/2}.
\]
Therefore, (\ref{phidiffdecomp}) is exponentially close to $\int f_{\Phi} d\mu$.


\end{proof}

We now consider the general case where $t$ and $b$ are globally
admissible periodic (but not necessarily constant) rows with common period $p$ and generalize
Theorem~\ref{main_Gibbs} to this case. Since such $t$ and $b$ always exist, this
will yield a way to efficiently approximate
$P(f_{\Phi})$ for any translation-invariant n.n. interaction $\Phi$ on a strongly irreducible
n.n. $\zz^2$-SFT $X$ for which $q(\Lambda_{\Phi,X}) < p_c$.


Let $X^{[p]}$ denote $X^{[[0,p-1] \times \{0\}]}$, the $([0,p-1] \times \{0\})$-higher power code of $X$, which is a n.n. $\zz^2$-SFT over the alphabet $\A^{[p]} := \L_{[0,p-1]\times \{0\}}(X)$.
Define a
new translation-invariant n.n. interaction $\Phi^{[p]}$ on $X^{[p]}$ as
follows:
\begin{itemize}
\item
On vertices: $\Phi^{[p]}([x_0 \ldots, x_{p-1}]) = \sum_{i=0}^{p-1} \Phi(x_i)$
\item
On horizontal edges:  $\Phi^{[p]}([x_0 \ldots, x_{p-1}] [y_0 \ldots, y_{p-1}] ) =
\Phi(x_0x_1) + \ldots + \Phi(x_{p-1}y_{0})$
\item
On vertical edges: $\Phi^{[p]}\left(\substack{[x_0 \ldots, x_{p-1}]\\[4pt] [y_0 \ldots, y_{p-1}]  }\right)   = \Phi\left(\substack{x_0\\[4pt] y_0}\right) + \ldots +
\Phi\left(\substack{x_{p-1}\\[4pt] y_{p-1}}\right)$
\end{itemize}
Let $t^{[p]}, b^{[p]}$ be the constant rows $(t|_{[0,p-1]})^\zz,
(b|_{[0,p-1]})^\zz \in (\A^{[p]})^{\zz}$. Then $t^{[p]}, b^{[p]}$ are globally
admissible in $X^{[p]}$.

If $q(\Lambda_{\Phi^{[p]}, X^{[p]}}) < p_c$, we can simply apply Theorem~\ref{main_Gibbs} to
achieve exponentially converging approximations to $P(f_{\Phi^{[p]}})$, and it is fairly
easy to see that $P(f_{\Phi^{[p]}}) = pP(f_{\Phi})$, so we would be done.
However, by examining the definitions, we see that it could be the case that
$q(\Lambda_{\Phi^{[p]}, X^{[p]}}) > q(\Lambda_{\Phi,X})$, and so we must take a more circuitous route.

%

As earlier, we can define $\overline{\Phi^{[p]}}_{m,n,t^{[p]},b^{[p]}}$ on $(X^{[p]})_{m,n,t^{[p]},b^{[p]}}$
and the corresponding matrix $A_{\Phi^{[p]},X^{[p]},m,n,t^{[p]},b^{[p]}}$. For fixed $X,\Phi,t,b,p$ and any
$n$ greater than the filling distance of $X^{[p]}$, we make the notation
$A^{[p]}_n := A_{\Phi^{[p]},X^{[p]},1,n,t^{[p]},b^{[p]}}$ and $\lambda^{[p]}_n := \lambda(A_n^{[p]})$, the largest
eigenvalue of $A^{[p]}$.

\begin{theorem}
\label{main_Gibbs2}
Let $X$ be a strongly irreducible
n.n. $\zz^2$-SFT and $\Phi$ a translation-invariant
n.n. interaction on $X$. Let $t$ and $b$ be globally
admissible periodic rows, with common period $p$. Assume
that $q(\Lambda_{\Phi,X}) < p_c$.
Then
there exist constants $Q, R >0$ such that for sufficiently large $n$,
$$
|(1/p)(\log \lambda^{[p]}_{n+1} - \log \lambda^{[p]}_{n}) - P(f_{\Phi})| < Qe^{-Rn}.
$$
\end{theorem}

\begin{proof}
Define $\Lambda^{[p]} = \Lambda_{\Phi^{[p]},X^{[p]}}$. We will now
verify that $q(\Lambda^{[p]}) < 1$. Any boundary configurations
$\delta^{[p]},\delta'^{[p]} \in \L_{N_{(0,0)}}(X^{[p]})$ correspond to configurations
$\delta, \delta' \in \L_{\partial([0,p-1] \times \{0\})}(X)$. By Proposition~\ref{HZ}
(which we may use because $q(\Lambda_{\Phi,X}) < p_c < 1$),
any such $\delta,\delta'$ have a common globally admissible filling in $X$, which implies that
$\delta^{[p]},\delta'^{[p]}$ have a common globally admissible filling in $X^{[p]}$,
which has positive $(\Lambda^{[p]})^{\delta^{[p]}}$, $(\Lambda^{[p]})^{\delta'^{[p]}}$-values
since $\Lambda^{[p]}$ is a Gibbs specification.

Therefore, by Proposition~\ref{stripuniqueness}, for large enough $n$,
there is a unique induced Gibbs state $\mu_{\Phi^{[p]},X^{[p]},1,n,t^{[p]},b^{[p]}}$ for $X^{[p]}$ and $\Phi^{[p]}$.
By again using Proposition~\ref{Markov} and (\cite[p. 99]{krieger}, \cite{blascamp}),
$\mu_{\Phi^{[p]},X^{[p]},1,n,t^{[p]},b^{[p]}}$ is the unique equilibrium state
for $\overline{\Phi^{[p]}}_{1,n,t^{[p]},b^{[p]}}$ on
$(X^{[p]})_{1,n,t^{[p]},b^{[p]}}$ and
\begin{multline} \label{lp}
\log \lambda^{[p]}_n =
P(\overline{\Phi^{[p]}}_{1,n,t^{[p]},b^{[p]}}) =
h(\mu_{\Phi^{[p]},X^{[p]},1,n,t^{[p]},b^{[p]}}) \\+ \int
\overline{\Phi^{[p]}}_{1,n,t^{[p]},b^{[p]}} d\mu_{\Phi^{[p]},X^{[p]},1,n,t^{[p]},b^{[p]}}.
\end{multline}

We now must show that the differences of the right-hand sides of (\ref{lp}) for $n$ and $n+1$ in fact converge exponentially fast to $pP(f_{\Phi})$, which is done in much the same way as in the proof of Theorem~\ref{main_Gibbs}. We again take an equilibrium state $\mu$ for $f_{\Phi}$ and $X$,
which is a Gibbs state for $\Phi$ on $X$ and is the unique $\mathbb{Z}^2$-MRF associated to the
$\mathbb{Z}^2$-specification $\Lambda_{\Phi,X}$. Since $q(\Lambda_{\Phi,X}) < 1$, we may define the unique MRF
$\mu_{\Phi,X,1,n,t,b}$ associated to $\Lambda_{\Phi,X,1,n,t,b}$.

We claim that
$\mu_{\Phi^{[p]},X^{[p]},1,n,t^{[p]},b^{[p]}} = (\mu_{\Phi,X,1,n,t,b})^{[p]}$.
This follows from Proposition~\ref{stripuniqueness} and the fact that $q(\Lambda^{[p]}) <1$
once one verifies that $(\mu_{\Phi,X,1,n,t,b})^{[p]}$ is associated to
$(\Lambda^{[p]})_{1,n,t^{[p]},b^{[p]}} = \Lambda_{\Phi^{[p]},X^{[p]},1,n,t^{[p]},b^{[p]}}$.
This is straightforward (but a bit tedious), and we leave the details to the reader. Thus, (\ref{lp}) becomes
$$
\log \lambda^{[p]}_n =
h(({\mu_{\Phi,X,1,n,t,b}})^{[p]}) + \int \overline{\Phi^{[p]}}_{1,n,t^{[p]},b^{[p]}}
d(({\mu_{\Phi,X,1,n,t,b}})^{[p]}).
$$
By Theorem~\ref{maintheorem2}, $h(({\mu_{\Phi,X,1,n+1,t,b}})^{[p]}) -
h(({\mu_{\Phi,X,1,n,t,b}})^{[p]})$ converges exponentially to
$ph(\mu)$.  Arguing the same way as in the proof of
Theorem~\ref{main_Gibbs}, we see that
\[
\int \overline{\Phi^{[p]}}_{1,n+1,t^{[p]},b^{[p]}} d(({\mu_{\Phi,X,1,n+1,t,b}})^{[p]}) -
\int \overline{\Phi^{[p]}}_{1,n,t^{[p]},b^{[p]}} d(({\mu_{\Phi,X,1,n,t,b}})^{[p]})
\]
converges exponentially to $\int f_{\Phi^{[p]}} d\mu = \int \sum_{i=0}^{p-1} (f_{\Phi} \circ \sigma_{(i,0)}) d\mu = p \int f_{\Phi} d\mu$, where the latter equality comes from translation-invariance of $\mu$.

Thus, $(1/p)(\log \lambda^{[p]}_{n+1} -   \log \lambda^{[p]}_{n})$
converges exponentially to $P(f_{\Phi})$, as desired.

\end{proof}

We will make a brief aside here to consider the utility of this theorem. For a n.n. $\zz^2$-SFT $X$ and translation-invariant n.n. interaction $\Phi$ with associated Gibbs specification $\Lambda_{\Phi,X}$, recall that $\Lambda^{\delta}_{\Phi,X}$ is defined by a formula involving $X$ and $\Phi$ for every $\delta \in \L_{N_{(0,0)}}(X)$, and then for any other $\delta' \in \A^{N_{(0,0)}}$, $\Lambda^{\delta'}_{\Phi,X}$ is just defined to match one of the existing $\Lambda^{\delta}_{\Phi,X}$. This means that $q(\Lambda_{\Phi,X})$ is, in reality, a minimum variational distance between $\Lambda^{\delta}_{\Phi,X}$ and $\Lambda^{\delta'}_{\Phi,X}$ for globally admissible $\delta, \delta'$ only.

It is well known (\cite{berger}) that checking whether or not a given configuration is globally admissible in a n.n. $\zz^2$-SFT can be undecidable, and so certainly algorithmically impossible. However, it is shown in Corollary 3.5 of \cite{HM} that for a strongly irreducible n.n. $\zz^2$-SFT, global admissibility of a
configuration is algorithmically checkable.

In practice however, this checking process can be very time-consuming, and so it is often easier to consider a ``simpler'' version of $q(\Lambda_{\Phi,X})$.
Assume that $\Phi$ is defined on all of $\A \cup \E_1 \cup \E_2$.
Say that a configuration $\delta \in \A^{N_{(0,0)}}$ is {\em fillable} if there exists
$x\in \A^{\{(0,0)\}}$ such that $x\delta$ is locally admissible.  For fillable $\delta$,
the formula (\ref{Gibbs_spec}) makes sense and we can define
$$
\widehat{q}(\Lambda_{\Phi,X})=
\max d(\Lambda_{\Phi,X}^{\delta}, \Lambda_{\Phi,X}^{\delta'})
$$
where the $\max$ is taken over only fillable configurations
$\delta, \delta' \in \A^{N_{(0,0)}}$.
Then $q(\Lambda_{\Phi,X}) \leq \widehat{q}(\Lambda_{\Phi,X})$, and so if
$ \widehat{q}(\Lambda_{\Phi,X}) < p_c$, $q(\Lambda_{\Phi,X}) < p_c$ as well. Since computing
 $\widehat{q}(\Lambda_{\Phi,X})$
 only requires finding the set of locally admissible configurations in $X$ with shape $\{(0,0)\} \cup N_{(0,0)}$, it is a far easier quantity to find.

%
%

\section{Computability}\label{computability}

We now address the issue of how efficiently our methods can be used to approximate $P(f_{\Phi})$ for a function $f_{\Phi}$ induced by a n.n. interaction $\Phi$ on a strongly irreducible n.n. SFT $X$.

We define $\alpha \in \mathbb{R}$ to be a \textit{computable number} if there exists a Turing machine $T$ which, on input $n$, outputs a number $\frac{p_n}{q_n} \in \mathbb{Q}$ such that $\displaystyle \Big|\alpha - \frac{p_n}{q_n}\Big| < 2^{-n}$. For any sequence of positive integers $\{r_n\}$, we say that $\alpha$ is $\{r_n\}$-computable if there exists such a Turing machine $T$ which computes $\frac{p_n}{q_n}$ in less than $r_n$ steps for all sufficiently large $n$. (For more information on computability theory, see \cite{Ko}.)

We say that $\Phi$ is $\{r_n\}$-computable if each value $\Phi(w)$ in the range of $\Phi$ is $\{r_n\}$-computable.

\begin{theorem}\label{pressurecomputability}
For any $\{r_n\}$-computable $\Phi$ and strongly irreducible n.n. $\zz^2$-SFT $X$ for which $q(\Lambda_{\Phi,X}) < p_c$, there exist constants $B$, $C$, and $J$ such that $P(f_{\Phi})$ is $\{J^n + B r_{Cn}\}$-computable.
\end{theorem}

\begin{proof}
We will not include every detail of the argument, but just describe the algorithm for approximating $P(f_{\Phi})$ and summarize the most computationally intensive steps. For a similar argument with more details included, see \cite{Pa}.

Some preprocessing must be done before any approximations. Firstly, a globally admissible periodic row $t$ for $X$ must be found; a careful reading of the proof in \cite{ward} shows that this can be done algorithmically. Denote by $p$ the period of $t$. Also, we invest a finite number of steps to find explicit integers $L$ and $U$ which bound all $\Phi(w)$ from below and above respectively. This finite amount of computation is negligible compared to the computation times in the theorem, and so we may safely ignore it.

Then, by Theorem~\ref{main_Gibbs2}, there exists $R>0$ such that for sufficiently large $n$,
\[
\left|1/p\left(\log \lambda(A_{\Phi^{[p]},X^{[p]},1,n+1,t^{[p]},t^{[p]}}) - \log \lambda(A_{\Phi^{[p]},X^{[p]},1,n,t^{[p]},t^{[p]}})\right) - P(f_{\Phi})\right| < 0.2e^{-Rn}.
\]
To approximate $P(f_{\Phi})$ to within $2^{-n}$, it then clearly suffices to approximate \newline $\lambda(A_{\Phi^{[p]},X^{[p]},1,k,t^{[p]},t^{[p]}})$ and $\lambda(A_{\Phi^{[p]},X^{[p]},1,k+1,t^{[p]},t^{[p]}})$ to within $0.4 \cdot 2^{-n}$, where $k = n \lceil \frac{\log 2}{R} \rceil$. It obviously suffices to describe the procedure for $\lambda(A_{\Phi^{[p]},X^{[p]},1,k,t^{[p]},t^{[p]}})$. We from now on refer to $A_{\Phi^{[p]},X^{[p]},1,k,t^{[p]},t^{[p]}}$ simply as $A$ for ease of reading.

The entries of $A$, indexed by legal columns $c$ and $d$, are all of the form $\displaystyle e^{-\sum \Phi(w)}$, where the sum is always over a set of at most $4kp$ configurations which are easily computed in polynomial (negligible) time in $k$, given $t$, $c$, and $d$. Therefore, the smallest nonzero entry of $A$ is at least $e^{-4kpU}$ and the largest entry is at most $e^{-4kpL}$. We now wish to approximate each entry of $A$ to within a tolerance of $0.2 |\A|^{-kp} e^{-4kpU} e^{8kpL} 2^{-n}$. We first approximate each individual $\Phi(w)$ to within
$\frac{1}{40kp} |\A|^{-kp} e^{-4kpU} e^{8kpL} 2^{-n}$. Since $k$ is linear in $n$, this expression is only exponentially small in $n$. Since each $\Phi(w)$ is $\{r_n\}$-computable, there exists $C$ so that such an approximation can be found for each $\Phi(w)$ in fewer than $r_{Cn}$ steps for sufficiently large $n$, and so a collection of such approximations for all $\Phi(w)$ can be found in fewer than $Br_{Cn}$ steps, where $B$ is the constant number of configurations $w$ for which $\Phi(w) \neq 0$. For each entry
$A_{c,d} = e^{-\sum \Phi(w)}$ of $A$, we then have an approximation to $-\sum \Phi(w)$ within a tolerance of $0.1 |\A|^{-kp} e^{-4kpU} e^{8kpL} 2^{-n}$, which yields an approximation to $A_{c,d} = e^{-\sum \Phi(w)}$ to within a tolerance of $0.2 |\A|^{-kp} e^{-4kpU} e^{4kpL} 2^{-n}$ for large enough $n$ since $A_{c,d} < e^{-4kpL}$.

We then have a matrix $\widetilde{A}$ in which each entry is within $0.2 |\A|^{-kp} e^{-4kpU} e^{4kpL} 2^{-n}$ of the corresponding entry of $A$. Since this matrix has only exponentially many entries (the size of $A = A_{\Phi^{[p]},X^{[p]},1,k,t^{[p]},t^{[p]}}$ is at most the size of the alphabet of $X^{[p]}_{1,k,t^{[p]},t^{[p]}}$, which is at most $|\A|^{kp}$), and since each approximation only involves summing previously recorded approximations to $\Phi(w)$ and exponentiating, there exists $F$ so that $\widetilde{A}$ can be computed in fewer than $F^n + Br_{Cn}$ steps for sufficiently large $n$.

Then $(1 - 0.2 |\A|^{-kp} e^{4kpL} 2^{-n}) A < \widetilde{A} < (1+0.2 |\A|^{-kp} e^{4kpL} 2^{-n}) A$, and by monotonicity of the Perron eigenvalue, $(1-0.2 |\A|^{-kp} e^{4kpL} 2^{-n}) \lambda(A) < \lambda(\widetilde{A}) < (1+0.2 |\A|^{-kp} e^{4kpL} 2^{-n}) \lambda(A)$. Since clearly $\lambda(A)$ is bounded from above by the maximum row sum of $A$, which itself is less than $|\A|^{kp} e^{-4kpL}$, this means that $|\lambda(\widetilde{A}) - \lambda(A)| < 0.2 \cdot 2^{-n}$.

All that remains is to approximate $\lambda(\widetilde{A})$ to within a tolerance of $0.2 \cdot 2^{-n}$. Since $\widetilde{A}$ and $A$ have the same nonzero entries, and since $A$ is primitive, there exists $N = N(k)$ such that $\widetilde{A}^N$ has all positive entries. We assume that $N$ is the smallest such integer, which is called the index of primitivity of $A$. It is well-known (\cite{HJ}, Corollary 8.5.9) that the index of primitivity is at most quadratic in the size of the matrix, so there exists $G$ independent of $n$ so that $N < G^k$. Clearly the smallest entry of $\widetilde{A}^N$, call it $\epsilon$, is at least $(e^{-5kpU})^N$. (We changed $4kpU$ to $5kpU$ to account for the fact that the smallest entry of $\widetilde{A}$ could be slightly smaller than the smallest entry of $A$.)

The reader can verify that for any $M$ and $k$,
\[
\Big(\epsilon \sum (\widetilde{A}^M)_{c,d}\Big)^k < \sum (\widetilde{A}^{k(M+N)})_{c,d} < \Big(\sum (\widetilde{A}^{M+N})_{c,d}\Big)^k.
\]
By taking logs, dividing by $k(M+N)$, and letting $k \rightarrow \infty$, we see that
\[
\frac{\log \epsilon}{M+N} + \frac{\log \sum (\widetilde{A}^M)_{c,d}}{M+N} < \lambda(\widetilde{A}) < \frac{\log \sum (\widetilde{A}^{M+N})_{c,d}}{M+N}.
\]
If we denote $\displaystyle f_M = \frac{\log \sum (\widetilde{A}^M)_{c,d}}{M}$, then for every $M>N$,
\[
\lambda(\widetilde{A}) < f_M < \lambda(\widetilde{A}) + \frac{N\lambda(\widetilde{A})}{M} - \frac{\log \epsilon}{M}.
\]
For $f_M$ to approximate $\Lambda(\widetilde{A})$ to within $0.2 \cdot 2^{-n}$, it is therefore sufficient to take $M > 5 \cdot 2^n (N\lambda(\widetilde{A}) - \log \epsilon)$, which is less than $H^n$ for some constant $H$ and large enough $n$. The calculation of $f_{H^n}$ entails taking an exponentially large power of an exponentially large matrix,
which can be done in exponentially many computations. Therefore, there exists $I$ so that $\lambda(\widetilde{A})$ can be approximated to within $0.2 \cdot 2^{-n}$ in fewer than $I^n$ computations.

By collecting all of these facts and taking $J = \max(I,F) + 1$, we see that for sufficiently large $n$, we may approximate $P(f_{\Phi})$ to within $2^{-n}$ by performing fewer than $J^n + B r_{Cn}$ steps for some uniform constants $B$, $C$, and $J$.

\end{proof}

\section{Examples}\label{examples}

We now give some applications of the results from Sections~\ref{pressure} and \ref{computability} to specific Gibbs states and pressures, beginning with the hard-core and Ising antiferromagnetic models presented earlier. We note that these models have strongly irreducible underlying SFTs (the hard square shift and full shift on $\pm 1$ respectively), and so checking whether our results apply to them boils down to checking which parameter values give $q(\Lambda_{\Phi,X}) < p_c$. 

\medskip

1. Hard-core model: we wish to know which activity levels $a$ give $q(\Lambda_{\Phi,X}) < p_c$. By the definition of $\Phi$, it is easy to check that for $\delta = 0^{N_{(0,0)}}$,
\[
\Lambda^{\delta}_{\Phi,\H}(b) =
\begin{cases}
\frac{1}{1+a} & \textrm{if } b = 0\\
\frac{a}{1+a} & \textrm{if } b = 1,
\end{cases}
\]
which we abbreviate by $\Lambda^{\delta}_{\Phi,\H} = \left(\frac{1}{1+a}, \frac{a}{1+a}\right)$. For any other $\delta' \in \{0,1\}^{N_{(0,0)}}$, $\Lambda^{\delta'}_{\Phi,\H} = (1,0)$, i.e. it is concentrated entirely on the $0$ symbol. (This is because if $\delta' \neq 0^{N_{(0,0)}}$, the only locally admissible way to fill $\delta$ in $\mathcal{H}$ is with a $0$.) Therefore,
\[
q(\Lambda_{\Phi,\H}) = d\left(\left(\frac{1}{1+a},\frac{a}{1+a}\right),(1,0)\right) = \frac{a}{1+a},
\]
which is less than $p_c$ iff $a < \frac{p_c}{1-p_c}$. We note that this computation was already done in \cite{vdBM}.

\medskip

2. Ising antiferromagnet: again, we wish to know which parameters $h, \beta$ give $q(\Lambda_{\Phi,X}) < p_c$. Again, it is reasonably straightfoward to check that for any $\delta \in \{\pm 1\}^{N_{(0,0)}}$ with $\sum_{v \in N_{(0,0)}} \delta(v) = n$,
\[
\Lambda^{\delta}_{\Phi,X}(b) =
\begin{cases}
\frac{e^{-\beta(h-n)}}{e^{\beta(h-n)} + e^{-\beta(h-n)}} & \textrm{if } b = -1\\
\frac{e^{\beta(h-n)}}{e^{\beta(h-n)} + e^{-\beta(h-n)}} & \textrm{if } b = 1.
\end{cases}
\]
It is then fairly easy to see that
\begin{multline*}
q(\Lambda_{\Phi,X}) = d((\Lambda_{\Phi,X})^{(1^{N_{(0,0)}})}, (\Lambda_{\Phi,X})^{(-1^{N_{(0,0)}})}) \\= \frac{e^{\beta(h-4)}}{e^{\beta(h-4)} + e^{-\beta(h-4)}} - \frac{e^{\beta(h+4)}}{e^{\beta(h+4)} + e^{-\beta(h+4)}}.
\end{multline*}
It was shown in \cite{vdBM} that $q(\Lambda_{\Phi,X}) < p_c$ as long as $2\beta(4-|h|) < \log \left(\frac{p_c}{1-p_c}\right)$, though this condition is certainly not necessary.

\medskip

In both models, the results of \cite{vdBM} imply that there is a unique Gibbs measure $\mu$ for the described parameters. Theorem~\ref{main_Gibbs2} implies that in addition, the pressure $P(f_{\Phi})$ is exponentially well approximable by
\[
1/p\left(\log \lambda(A_{\Phi^{[p]},X^{[p]},1,n+1,t^{[p]},b^{[p]}}) - \log \lambda(A_{\Phi^{[p]},X^{[p]},1,n,t^{[p]},b^{[p]}})\right)
\]
for any boundary conditions $t,b$ periodic with period $p$. (Of course, for these two models, there exist globally admissible constant rows, and so we could take $p=1$.) By Theorem~\ref{pressurecomputability}, a bound on the computability of these pressures can also be given in terms of the computability of the relevant parameter values.

\medskip
We conclude by giving applications to topological entropy, in the same spirit as \cite{Pa}. 
We first note that since a measure of maximal entropy is clearly an equilibrium state for the function $f=f_0=0$, any measure of maximal entropy is a $\mathbb{Z}^2$-MRF associated to the Gibbs specification $\Lambda = \Lambda_{0,X}$ defined by
\[
\Lambda^{\delta}(\eta) :=
\begin{cases}
0 & \textrm{ if } \quad \eta \delta \notin \L(X)\\
\frac{1}{N(\delta)} & \textrm{ if } \quad \eta \delta \in \L(X)
\end{cases}
\]
when $\delta \in \L(X)$ (here $N(\delta)$ is just the normalization factor $|\{\eta \in \A^S \ : \ \eta \delta \in \L(X)\}|$), and by $\Lambda^{\delta'} = \Lambda^{\delta_0}$ when $\delta' \notin \L(X)$, where $\delta_0$ is any fixed globally admissible boundary with the same shape as $\delta'$.

We now exhibit two classes of strongly irreducible n.n. $\zz^2$-SFTs for which $q(\Lambda) < p_c$ for this specification, implying that the topological entropy (topological pressure for $f = f_0 = 0$) is exponentially well approximable via strips, and therefore $\{J^n\}$-computable for some $J$. For the first, we need a definition.

\begin{definition}
In a n.n. $\zz^d$-SFT $X$ with alphabet $\A$, $a \in \A$ is a {\rm safe symbol} if for all $b \in \A$, $(a,b), (b,a) \in \mathcal{E}_i$ for $1 \leq i \leq d$. In other words, $a$ is a safe symbol if it may legally appear next to any letter of the alphabet in any direction.
\end{definition}

\begin{proposition}\label{manysafe} Any n.n. $\mathbb{Z}^2$-SFT with alphabet $\A$ containing a subset $\A'$ of safe symbols for which $\frac{|\A'|}{|\A|} > 1 - p_c$ has a unique measure of maximal entropy $\mu$ which is an MRF associated to a translation-invariant $\mathbb{Z}^2$-specification $\Lambda$ satisfying $q(\Lambda) < p_c$. The topological entropy of any such SFT is $\{J^n\}$-computable for some $J$.
\end{proposition}

\begin{proof}
Consider any such $\mathbb{Z}^2$-SFT $X$ and any $\mu$ a measure of maximal entropy on $X$ with $\mathbb{Z}^2$-specification $\Lambda$. Clearly, since $X$ contains a safe symbol,
any $\delta \in A^{N_{(0,0)}}$ is in $\L(X)$ and
$X$ is strongly irreducible. For any $\delta \in \A^{N_{(0,0)}}$, the probability distribution $\Lambda^{\delta}$ is uniform over some subset of $\A$, call it $S_{\delta}$, which contains $\A'$. Clearly, $|S_{\delta}| = N(\delta)$. Choose any $\delta,\delta' \in \A^{N_{(0,0)}}$, and assume w.l.o.g. that $N(\delta) \geq N(\delta')$. Then,
\[
d\big(\Lambda^{\delta},\Lambda^{\delta'}\big) = \frac{1}{2} \sum_{e \in \A} |\Lambda^{\delta}(e) - \Lambda^{\delta'}(e)| \leq
\]
\[
\frac{|\A'|}{2} \Big(\frac{1}{N(\delta')} - \frac{1}{N(\delta)}\Big) + \frac{N(\delta) - |\A'|}{2N(\delta)} + \frac{N(\delta') - |\A'|}{2N(\delta')}
\]
\[
= 1 - \frac{|\A'|}{N(\delta)} \leq 1 - \frac{|\A'|}{|\A|} < p_c.
\]

Therefore, $q(\Lambda) < p_c$, implying by Theorem~\ref{uniqueness} that $\mu$ was unique. The fact that $h(X) = P(f_0)$ implies that $h(X)$ is $\{J^n\}$ computable for some $J$ by Theorem~\ref{pressurecomputability}.

\end{proof}

We note that since $p_c > .556 > .5$ by \cite{vdBE}, clearly the $\mathbb{Z}^2$ hard square shift $\mathcal{H}$ satisfies the conditions of Proposition~\ref{manysafe}.

We recall that the $\{J^n\}$-computability of the topological entropy of such SFTs followed from the exponentially good approximations given by differences of topological entropies of constrained strips, i.e. strips with boundary conditions $t,b$. For SFTs with at least one safe symbol $a$ (such as those to which Proposition~\ref{manysafe} applies), one can take $t,b = a^{\zz}$ and then the topological entropies of the approximating ``constrained'' strips are equal to those of unconstrained strips as were treated in \cite{Pa}.

\begin{proposition}\label{manyadj}
Any n.n. $\mathbb{Z}^2$-SFT with alphabet $\A$ with the property that for all $a \in \A$, and for any direction, the set of legal neighbors of $a$ in that direction has cardinality greater than $(1 -  \frac{p_c}{4(1+p_c)})|\A|$, has a unique measure of maximal entropy $\mu$ which is an MRF associated to a $\mathbb{Z}^2$-specification $\Lambda$ satisfying $q(\Lambda) < p_c$. The topological entropy of any such SFT is $\{J^n\}$-computable for some $J$.
\end{proposition}

\begin{proof}
Consider any such n.n. $\mathbb{Z}^2$-SFT $X$ and any $\mu$ a measure of maximal entropy on $X$ with $\mathbb{Z}^2$-specification $\Lambda$. Since $1 - \frac{p_c}{4(1+p_c)} > \frac{3}{4}$, any $\delta \in \A^{N_{(0,0)}}$ can be extended to a locally admissible configuration with shape $N_{(0,0)} \cup \{0\}$. The reader may check that this implies both that any $\delta \in A^{N_{(0,0)}}$ is in $\L(X)$ and that $X$ is strongly irreducible. Therefore, for any $\delta \in A^{N_{(0,0)}}$, the probability distribution $\Lambda^{\delta}$ is uniform over some nonempty subset of $\A$, call it $S_{\delta}$. Clearly $|S_{\delta}| = N(\delta)$. Define $\alpha := \frac{p_c}{(1+p_c)}$. Then it is clear that $|S_{\delta}| > (1-\alpha)|\A|$
for any $\delta \in \A$. Choose any $\delta,\delta' \in \A^{N_{(0,0)}}$, and assume w.l.o.g. that $N(\delta) \geq N(\delta')$. Define $m := |S_{\delta} \cap S_{\delta'}|$ and note that $m > N(\delta) - \alpha |\A|$. Then,
\begin{multline*}
d\big(\Lambda^{\delta},\Lambda^{\delta'}\big) = \frac{1}{2} \sum_{e \in \A} |\Lambda^{\delta}(e) - \Lambda^{\delta'}(e)| \\
\leq \frac{m}{2} \Big(\frac{1}{N(\delta')} - \frac{1}{N(\delta)}\Big) + \frac{N(\delta) - m}{2N(\delta)} + \frac{N(\delta') - m}{2N(\delta')} = 1 - \frac{m}{N(\delta)}\\
< 1 - \frac{N(\delta) - \alpha |\A|}{N(\delta)} = \frac{\alpha|\A|}{N(\delta)} < \frac{\alpha}{1-\alpha} = p_c.
\end{multline*}

Therefore, $q(\Lambda) < p_c$, implying by Theorem~\ref{uniqueness} that $\mu$ was unique. The fact that $h(X) = P(f_0)$ again implies that $h(X)$ is $\{J^n\}$-computable for some $J$ by Theorem~\ref{pressurecomputability}.

\end{proof}

We note that since by \cite{vdBE}, $\frac{p_c}{4(1+p_c)} >
\frac{.556}{4(1+.556)} > \frac{5}{56}$, the usual
$k$-\textit{checkerboard} $\mathbb{Z}^2$-SFT, with alphabet
$\{1,\ldots,k\}$ and forbidden list consisting of all pairs of
adjacent identical letters, satisfies the conditions of
Proposition~\ref{manyadj} when $k \geq 12$. (It may also satisfy
these conditions for smaller $k$ depending on the exact value of
$p_c$.)
\section*{Acknowledgments}\label{acks}
The authors thank David Brydges, Guangyue Han, Erez Louidor,
Christian Maes, and Benjy Weiss for helpful discussions.

\end{document}